\documentclass[review]{elsarticle}
\usepackage{lineno,hyperref}
\journal{arXiv}
\usepackage{amsmath, amsthm, amsfonts, amssymb, latexsym, epsfig}
\usepackage{float}
\usepackage{graphicx}
\usepackage{graphics}
\usepackage{epsfig}
\usepackage{anysize}
\usepackage{lipsum}
\usepackage{mathtools}
\usepackage{expl3}
\usepackage{multirow}
\usepackage{wrapfig}
\usepackage{expl3}
\usepackage{comment}
\expandafter\def\csname ver@l3regex.sty\endcsname{} 
\usepackage{coloring}
\usepackage{csquotes}
\usepackage{float}
\usepackage{makecell}
\marginsize{2.5cm}{3cm}{3cm}{3cm}
\newtheorem{theorem}{Theorem}[section]
\newtheorem{proposition}{Proposition}[section]
\newtheorem{lemma}{Lemma}[section]
\newtheorem{remark}{Remark}[section]

\usepackage{color,xcolor,soul}

\begin{document}

\begin{frontmatter}

\title{
Repeated-Root Constacyclic Codes of Length $3p^s$ over the Finite Non-Chain Ring $\frac{\mathbb{F}_{p^m}[u, v]}{\langle u^2, v^2, uv-vu\rangle}$ and their Duals\\
}

\author{Divya Acharya $^a$, Prasanna Poojary$^a$, Vadiraja Bhatta G R$^b$\footnote{vadiraja.bhatta@manipal.edu}}
 \address{$^{a}$ Department of Mathematics, Manipal Institute of Technology Bengaluru, Manipal Academy of Higher Education, Manipal, Karnataka, India\\
 $^b$ Department of Mathematics, Manipal Institute of Technology, Manipal Academy of Higher Education, Manipal, Karnataka, India}

\begin{abstract}
%\begin{comment}
    This study aims to determine the algebraic structures of $\alpha$-constacyclic codes of length  $3p^s$ over the finite commutative non-chain ring $\mathcal{R}=\frac{\mathbb{F}_{p^m}[u, v]}{\langle u^2, v^2, uv-vu\rangle}$, for a prime $p \neq 3.$ For the unit $\alpha$, we consider two different instances: when  $\alpha$ is a cube in  $\mathcal{R}$ and when it is not. Analyzing the first scenario is relatively easy. When  $\alpha$ is not a unit in  $\mathcal{R}$, we consider several subcases and determine the algebraic structures of constacyclic codes in those cases. Further, we also provide the number of codewords and the duals of $\alpha$-constacyclic codes.  
\end{abstract}

\end{frontmatter}

%\linenumbers

\section{Introduction}
	Constacyclic codes are significant from a theoretical and practical standpoint since they may be efficiently encoded with a simple shift register and they establish strong connections with algebra. The $\alpha$-constacyclic codes of length $n$ over the finite field $\mathbb{F}_q$ are identified as ideals $\langle \ell(x)\rangle$  of the ambient ring $\frac{\mathbb{F}_q[x]}{\langle x^n-\alpha \rangle}$ , where $\ell(x)$ is a divisor of $x^n-\alpha$ . The codes are referred to as simple root codes when the length of the code $n$ is relatively prime to the characteristic of the finite field $\mathbb{F}_q$. Otherwise, they are known as repeated-root codes, which Berman first investigated in 1967 \cite{berman1967semisimple}, followed by a series of papers \cite{massey1973polynomial}, \cite{falkner1979existence}, \cite{roth1986cyclic}.
 Van Lint \cite{van1991repeated}  and Castagnoli et al. \cite{castagnoli1991repeated} carried out the most extensive study on repeated-root codes.
	
 Constacyclic codes over $\mathbb{Z}_4$ were studied in \cite{calderbank1996cyclic}, \cite{helleseth2001codes}.
    Constacyclic codes over $\mathbb{F}_{2} + u\mathbb{F}_{2}$ are very interesting because the structure of $\mathbb{F}_{2} + u\mathbb{F}_{2}$ is lies between $\mathbb{F}_{2^2}$ and $\mathbb{Z}_4$ as it is additively analogous to $\mathbb{F}_{2^2}$ and multiplicatively analogous to $\mathbb{Z}_4$ \cite{abualrub2009constacyclic}. 

    As a result, cyclic and constacyclic codes over $\mathbb{F}_{2} + u\mathbb{F}_{2}$ are of interest to many researchers.
    
	The structure of $\lambda$-constacyclic codes of lengths $p^s$ and $mp^s$ over the chain ring $\frac{\mathbb{F}_{p^m}[u]}{\langle u^k \rangle}$ respectively, was examined by Dinh et al. \cite{dinh2016repeated} and Guenda and Gulliver \cite{guenda2015repeated}. In \cite{kai20101+}, Kai et al. discused $(1+\lambda u)$-constacyclic codes of arbitrary length over the ring $\frac{\mathbb{F}_{p}[u]}{\langle u^m \rangle}$, for a unit $\lambda$ in $\mathbb{F}_{p}$. Constacyclic codes and their duals over the ring $\mathbb{F}_{p^m}+u\mathbb{F}_{p^m}$ were obtained in a series of papers \cite{dinh2010constacyclic}, \cite{dinh2015negacyclic}, \cite{chen2016constacyclic}, \cite{dinh2020constacyclic3ps}, \cite{dinh2018negacyclic}, \cite{dinh2018cyclic}, \cite{dinh2019alpha}, \cite{dinh2019class}. Constacyclic codes over the ring $\mathbb{F}_{p^m}+u\mathbb{F}_{p^m}+u^2\mathbb{F}_{p^m}$ were extensively studied in a series of papers \cite{sobhani2015complete}, \cite{sriwirach2021repeated}, \cite{boudine2023complete}, \cite{Complete2023Boudine}, \cite{laaouine2024class},\cite{36}.
 
	\begin{comment}
	    
	In \cite{dinh2010constacyclic}, Dinh obtained algebraic structures of all constacyclic codes of length $p^s$ over the ring $\mathbb{F}_{p^m}+u\mathbb{F}_{p^m}$ and their duals. Dinh et al. \cite{dinh2015negacyclic} examined negacyclic codes of length $2p^s$ over $\mathbb{F}_{p^m}+u\mathbb{F}_{p^m}$. Chen et al. \cite{chen2016constacyclic} investigated the algebraic structures of all constacyclic codes and their duals of length $2p^s$  over $\mathbb{F}_{p^m}+u\mathbb{F}_{p^m}$. Later, Dinh et al. \cite{dinh2020constacyclic3ps} categorized  all  constacyclic codes of length $3p^s$  over $\mathbb{F}_{p^m}+u\mathbb{F}_{p^m}$.
	% The algebraic structures of cyclic, negacyclic \cite{dinh2018negacyclic}, \cite{dinh2018cyclic} and  constacyclic \cite{dinh2019alpha}, \cite{dinh2019class} codes of length $4p^s$ over $\mathbb{F}_{p^m}+u\mathbb{F}_{p^m}$ was obtained.
	The algebraic structures of cyclic, negacyclic and constacyclic codes of length $4p^s$ over $\mathbb{F}_{p^m}+u\mathbb{F}_{p^m}$ were obtained in a series of papers \cite{dinh2018negacyclic}, \cite{dinh2018cyclic}, \cite{dinh2019alpha}, \cite{dinh2019class}.
	\end{comment}
	In \cite{yildiz2011cyclic}, Yildiz and Karadeniz examined cyclic codes of odd length over the ring $\frac{\mathbb{F}_{2}[u, v]}{\langle u^2, v^2, uv-vu\rangle}$ which is not a chain ring. As the Gray images of these cyclic codes, they found a few good binary linear codes. In \cite{karadeniz20111+} Yildiz and Karadeniz studied $(1+v)$-constacyclic codes over the ring $\mathbb{F}_{2}+u\mathbb{F}_{2}+v\mathbb{F}_{2}+uv\mathbb{F}_{2}$ and they obtained cyclic codes over $\mathbb{F}_{2}+u\mathbb{F}_{2}$ as images of  $(1+v)$-constacyclic codes over $\mathbb{F}_{2}+u\mathbb{F}_{2}+v\mathbb{F}_{2}+uv\mathbb{F}_{2}$ under a natural Gray map.
	
 Negacyclic codes of odd length over $\frac{\mathbb{F}_{p}[u, v]}{\langle u^2, v^2, uv-vu\rangle}$ are studied by Ghosh \cite{ghosh2015negacyclic}. Bag \cite{bag2019classes} investigated $(\lambda_1 + u\lambda_2)$ and $(\lambda_1 + v\lambda_3)$-constacyclic codes of prime power length over $\frac{\mathbb{F}_{p^m}[u, v]}{\langle u^2, v^2, uv-vu\rangle}$. The authors of \cite{dougherty2012cyclic} considered the more general ring $\frac{\mathbb{F}_{2}[u_1,u_2,\ldots,u_k]}{\langle u_{i}^2, v_{j}^2, u_iv_j-v_ju_i\rangle}$. 
 
 These investigations were expanded to cyclic codes over the ring $\frac{\mathbb{F}_{2^m}[u, v]}{\langle u^2, v^2, uv-vu\rangle}$ by Sobhani and Molakarimi in \cite{sobhani2013some}. Dinh et al. in \cite{dinh2020constacyclicps} investigated the structures of all constacyclic codes of prime power length over the ring $\frac{\mathbb{F}_{p^m}[u, v]}{\langle u^2, v^2, uv-vu\rangle}$.

	The paper is structured in the following manner. Section 2 provides a summary of preliminary information. The primary findings of this investigation are outlined in Section 3. If $\alpha$ is a cube in the ring $\mathcal{R}=\mathbb{F}_{p^m} + u\mathbb{F}_{p^m} + v\mathbb{F}_{p^m} +uv\mathbb{F}_{p^m}$,  say $\alpha=\beta^3$, then $\alpha$-constacyclic codes examine under the two cases: $p^m \equiv 1 (\mod{3})$ and $p^m \equiv 2 (\mod{3})$. The case when  $\alpha$ is non-cube in the ring $\mathcal{R}$ requires more detailed analysis and it is handled in several subcases. An element $\alpha=\alpha_1 + \alpha_2 u + \alpha_3 v + \alpha_4 uv \in \mathcal{R}$ is a unit if and only if $\alpha_1$ is nonzero in $\mathbb{F}_{p^m}$. Since we consider unit in  $\mathcal{R}$, we always assume $\alpha_1 \ne 0$. Considering the other three coefficients in $\alpha$  being zero or non-zero, we obtain a total of eight subcases: $\alpha=\alpha_1 $, $\alpha=\alpha_1 + \alpha_2 u $, $\alpha=\alpha_1 +\alpha_3 v $, $\alpha=\alpha_1 + \alpha_4 uv $, $\alpha=\alpha_1 + \alpha_2 u + \alpha_3 v$, $\alpha=\alpha_1 + \alpha_3 v + \alpha_4 uv $, $\alpha=\alpha_1 + \alpha_2 u + \alpha_4 uv $, $\alpha=\alpha_1 + \alpha_2 u + \alpha_3 v + \alpha_4 uv $, where $\alpha_1, \alpha_2, \alpha_3, \alpha_4 \in \mathbb{F}^{*}_{p^m}$. 
 In section 4, we examine the case in which the unit $\alpha=\alpha_1 + \alpha_3 v + \alpha_4 uv$, where $\alpha_1, \alpha_3, \alpha_4 \in \mathbb{F}^*_{p^m}$ and the $\alpha$ is non-cube in $\mathbb{F}_{p^m} + u\mathbb{F}_{p^m} + v\mathbb{F}_{p^m} +uv\mathbb{F}_{p^m}$. In this case, the $\alpha$-constacyclic codes are categorized into four distinct types based on the ideals of the local ring $\frac{\mathcal{R}[x]}{\langle x^{3p^s}-\alpha \rangle}$.  Furthermore, we provide the number of codewords and duals of $\alpha$-constacyclic codes. The structure, number of codewords, and duals of $(\alpha_1 + \alpha_2 u +\alpha_3 v + \alpha_4 uv)$-constacyclic codes of length $3p^s$ over $\mathcal{R}$ are considered in section 5, where $\alpha_1, \alpha_2, \alpha_3 \in \mathbb{F}^*_{p^m}$ and $\alpha_4 \in \mathbb{F}_{p^m}$. Section 6 presents  $\alpha=\alpha_1 + \alpha_4 uv$ constacyclic codes as ideals of $\frac{\mathcal{R}[x]}{\langle x^{3p^s}-(\alpha_1 + \alpha_4 uv) \rangle}$, which has the maximal ideal $\langle (x^3-\alpha_0),  u, v \rangle $.

\section{Preliminaries}
%\begin{comment}
    Let $\Re$ be a finite commutative ring with identity 1 and let $I$ be an ideal. If only one element generates $I$, then it is referred to as a \textit{principal ideal}. If every ideal in the ring $\Re$ is principal,  then the ring is a \textit{principal ideal ring}. If $\Re$ has a unique maximum ideal,  then $\Re$ is referred to as a $local~ ring$. If the set of all ideals in  $\Re$ forms a chain under inclusion,  the ring is referred to as a $chain~ ring.$
	\begin{proposition}\cite{dinh2004cyclic}\label{prop2.1}
    For a finite commutative ring $\Re$ the following conditions are equivalent:
    \begin{enumerate}
        \item $\Re$ is a local ring and its maximal ideal  is principal;
        \item $\Re$ is a local principal ideal ring;
        \item $\Re$ is a chain ring.
    \end{enumerate}
    \end{proposition}
    Let $\Re$ be a finite commutative ring with identity. A code $\mathcal{C}$ of length $n$ over $\Re$ is a nonempty subset of $\Re^n$. An element of $\mathcal{C}$ is called a codeword. $\mathcal{C}$ is called a  linear code over $\Re$ if $\mathcal{C}$ is an $\Re$-submodule of $\Re^n$. Let $\alpha$ be a unit of $\Re$.  The $\alpha$-constacyclic shift $\sigma_{\alpha}$ on $\Re^n$ is  defined by
    \begin{center}
		$\sigma_{\alpha}(\zeta_0,  \zeta_1,\ldots,  \zeta_{n-1}) = (\alpha \zeta_{n-1},  \zeta_0,\ldots, \zeta_{n-2}).$
    \end{center}
    A code $\mathcal{C}$ is said to be $\alpha$-constacyclic if $\mathcal{C}$ is closed under the operator $\sigma_{\alpha}$.
    If  $\alpha$ is equal to 1(or -1), then the $\alpha$-constacyclic codes are referred to as cyclic (or negacyclic) codes. In the study of cyclic and constacyclic codes, it is useful to identify a vector  $v = (v_0, v_1, \ldots,  v_{n-1})$ in $\Re^n$ with the polynomial  $v(x) = v_ 0+ v_1x +\cdots+v_{n-1}x^{n-1}$. 
   
 It is well-known that constacyclic codes are precisely the ideals in a quotient ring   \cite{macwilliams1977theory}, \cite{huffman2010fundamentals}.
    \begin{proposition}\label{1}
        A linear code $\mathcal{C}$ of length $n$ over $\Re$ is an $\alpha$-constacyclic if and only if $\mathcal{C}$ is an ideal of $\frac{\Re[x]}{\langle x^n-\alpha \rangle}$.
    \end{proposition}
    \begin{proposition}\cite{dinh2010constacyclic}
        The dual of an $\alpha$-constacyclic code is an $\alpha^{-1}$-constacyclic code.
    \end{proposition}
    \begin{proposition}\cite{dinh2004cyclic}
        Let $p$ be a prime number and $\Re$ be a finite chain ring of size $p^m$. The number of codewords in any linear code $\mathcal{C}$ of length $n$ over $\Re$ is $p^k$, for some integer $k\in \{0, 1,\ldots,mn\}$. Moreover, the dual code $\mathcal{C}$ has $p^\ell$ codewords, where $k + \ell = mn$, i.e., $\vert \mathcal{C} \vert \cdot \vert \mathcal{C}^{\perp} \vert =\vert \Re \vert ^n$.
    \end{proposition}

Let $p \neq 3$ be any prime number and $m$ be a positive integer. Let $\mathcal{R} = \frac{\mathbb{F}_{p^m}[u, v]}{\langle u^2, v^2, uv-vu\rangle}=\mathbb{F}_{p^m} + u\mathbb{F}_{p^m} + v\mathbb{F}_{p^m} +uv\mathbb{F}_{p^m}$ where $u^2 = 0$,  $v^2 = 0$,  $uv = vu$. An element $\alpha=\alpha_1 + \alpha_2 u + \alpha_3 v + \alpha_4 uv \in \mathcal{R}$ is a unit if and only if $\alpha_1$ is nonzero in $\mathbb{F}_{p^m}$. Let $\alpha_1 \in \mathbb{F}^{*}_{p^m}$. Now $\alpha_1^{p^{tm}}=\alpha_1$ for any positive integer $t$. For positive integers $m$ and $s$,  there exist non-negative integers $q_0$ and $r_0$ such that $s= q_0 m+r_0$ with $0\leq r_0 \leq m-1$, by the division algorithm. Let $\alpha_0=\alpha_1^{p^{(q_0+1)m-s}}=\alpha_1^{p^{m-r_0}}$. Then ${\alpha_0}^{p^{s}}=\alpha_1^{p^{(q_0+1)m}}=\alpha_1$.

  \noindent For any unit $\alpha$ of $\mathcal{R}$, let 
  \begin{center}
        $\mathcal{R}_\alpha=\frac{\mathcal{R}[x]}{\langle x^{3p^s}-\alpha \rangle}$.
  \end{center}
  It follows from Proposition\ref{1} that $\alpha$-constacyclic codes of length $3p^s$ over $\mathcal{R}$ are ideals
of $\mathcal{R}_\alpha$.

    The annihilator of an ideal $I$ of a ring $\mathcal{R}_\alpha$ is also an ideal and is defined as:
    $\mathcal{A}(I) = \{a \in \mathcal{R}_\alpha | ab = 0,  \text{ for all } b \in I\}$. For any polynomial $f (x) = m_0 + m_1x +\cdots+ m_kx^k \in \mathcal{R}[x]$,  where $m_k\neq 0$,  the reciprocal polynomial of $f (x)$ is defined as $f (x)^* =x^kf(x^{-1})= m_k + m_{k-1}x +\cdots+ m_0x^k $. If $I$ is an ideal of $\mathcal{R}_\alpha$,  then it is easy to show that $I^* = \{f^*(x)|f (x) \in I\}$ is also an ideal of $\mathcal{R}_\alpha$.
    \begin{proposition}\cite{chen2016constacyclic}
        Let $\mathcal{C}$ be a constacyclic code of length $n$ over $\mathcal{R}$. Then the dual $\mathcal{C}^{\perp}$ of $\mathcal{C}$ is $\mathcal{A}(\mathcal{C})^*$.
    \end{proposition}
	\begin{lemma}\cite{chen2016constacyclic}\label{2.6}
        \begin{enumerate}
            \item If $deg(g(x)) \geq deg(h(x))$,  then $(g(x) + h(x))^{*} =g^{*}(x) + x^{deg(g(x))-deg(h(x))}h^{*}(x).$
            \item  $(g(x)h(x))^{*}=g^*(x)h^{*}(x).$
        \end{enumerate}
    \end{lemma}

   \noindent The following lemma will be useful in our analysis. 
   
    \begin{lemma}\cite{dinh2020constacyclic3ps}\label{3}
        Let $I = \langle g(x),  uh(x)\rangle $ be an ideal of $\frac{(\mathbb{F}_{p^m} + u\mathbb{F}_{p^m})[x]}{\langle x^{3p^s}-\alpha \rangle}$. Then $I^*=\{f^*(x)| f(x) \in I\}=\langle g ^*(x),  uh^*(x) \rangle$.
    \end{lemma}

   In analyzing constacyclic codes over $\mathcal{R}$, we find it useful to consider two main cases for the unit $\alpha$. One of the cases is easily analysed via the Chinese Remainder Theorem.  The other case is subdivided into eight subcases. We consider three of them in this paper.
   
	%Let $\alpha$ be a unit of $\mathbb{F}_{p^m}$.
	% We divide the units of  $\mathcal{R}$ in the following five categories:
	%	\begin{enumerate}
		%		\item $\alpha$
		%		\item $\alpha + \delta_1uv$
		%		\item $ \alpha + \gamma v + \delta uv$
		%		\item $ \alpha + \beta u + \delta uv$
		%		\item $\alpha + \beta u + \gamma v + \delta uv$
		%	\end{enumerate} 
%\end{comment}

\section{  $\alpha$ in the form of a cube in $\mathcal{R}$} 
    \noindent  Consider a unit $\alpha$ of $\mathbb{F}_{p^m}$. When $\alpha$ is cube in $\mathcal{R}$ we determine the $\alpha$-constacyclic code of length $3p^s$. As $\alpha$ is a cube in $\mathcal{R}$ there exists $\beta \in \mathcal{R}$ such that $\alpha=\beta^3$. In this section, we consider the following two cases.
    \subsection{Suppose $p^m \equiv 1 (\mod{3})$}
    Now 
    \begin{align*}
        x^{3p^s}-\alpha &=x^{3p^s}-\beta^3\\
        &=(x^{p^s}-\beta)(x^{2p^s}+\beta x^{p^s}+\beta^2)
    \end{align*}
    \begin{lemma} \label{3.1}\cite{dinh2020constacyclic3ps}
        Let $p^m \equiv 1 (\mod3).$ Then there exist $\delta, \gamma \in \mathbb{F}_{p^m}$ such that $\gamma \delta=1$ and $\delta+\gamma=-1$ 
    \end{lemma}
    \noindent By Lemma \ref{3.1}, we have 
    \begin{equation*}
        x^{3p^s}-\alpha =(x^{p^s}-\beta)(x^{p^s}-\delta\beta)(x^{p^s}-\gamma\beta).
    \end{equation*}
    Then by the Chinese Remainder Theorem, 
    \begin{equation*}
        \frac{\mathcal{R}[x]}{\langle x^{3p^s}-\alpha \rangle} \cong \frac{\mathcal{R}[x]}{\langle x^{p^s}-\beta  \rangle}\oplus\frac{\mathcal{R}[x]}{\langle x^{p^s}-\delta \beta \rangle}\oplus\frac{\mathcal{R}[x]}{\langle x^{p^s}-\gamma \beta \rangle}.
    \end{equation*}
    This means that a direct sum of a $\beta$-constacyclic code, $\delta\beta$-constacyclic code and $\gamma\beta$-constacyclic code of length $p^s$ over $\mathcal{R}$  can be used to represent each $\alpha$-constacyclic code of length $3p^s$. Note that the algebraic structures of all constacyclic codes of length $p^s$ over $\mathcal{R}$ are determined in \cite{dinh2020constacyclicps}. %Therefore, the structure of the constacyclic codes is determined in this case. 
    
    \subsection{Suppose $p^m \equiv 2 (\mod{3})$}
    \begin{lemma}\cite{dinh2020constacyclic3ps}\label{3.2}
        For every $c \in \mathbb{F}_{p^m}\backslash \{0\}$, the polynomial $x^2+cx+c^2$ is irreducible in $\mathbb{F}_{p^m}[x]$.
    \end{lemma}
    \begin{lemma}
        For every $c \in \mathbb{F}_{p^m}\backslash \{0\}$, the polynomial $x^2+cx+c^2$ is irreducible in $\mathcal{R}[x]$. 
    \end{lemma}
    \begin{proof}
        Suppose $x^2+cx+c^2$ is reducible polynomial in $\mathcal{R}[x].$ Then there exists $\lambda=\lambda_0 +\lambda_1u+ \lambda_2 v + \lambda_3 uv \in \mathcal{R}$ for some $\lambda_0, \lambda_1, \lambda_2, \lambda_3 \in \mathbb{F}_{p^m}$ such that $\lambda^2+c\lambda+c^2=0.$ Therefore,
        \begin{equation*}
            (\lambda_0 +\lambda_1u+ \lambda_2 v + \lambda_3 uv)^2+c(\lambda_0 +\lambda_1u+ \lambda_2 v + \lambda_3 uv)+c^2=0
            \end{equation*}
            \begin{equation*}(\lambda_0^2+a^2+a\lambda_0)+(2\lambda_0 \lambda_1+a \lambda_1)u+(2\lambda_0 \lambda_2+a \lambda_2)v+(2\lambda_0 \lambda_3+2 \lambda_1 \lambda_2+a \lambda_3)uv=0
        \end{equation*}
    This implies $\lambda_0^2+c^2+c\lambda_0=0$, $2\lambda_0 \lambda_1+c \lambda_1=0$, $2\lambda_0 \lambda_2+c \lambda_2=0$ and $2\lambda_0 \lambda_3+2 \lambda_1 \lambda_2+c \lambda_3=0$. Thus, $\lambda_0$ is a root of $x^2+cx+c^2$, which is a contradiction to the Lemma \ref{3.2}. Thus $x^2+cx+c^2$ is irreducible in $\mathcal{R}[x].$
    \end{proof}

\noindent Therefore  $x^{3p^s}-\alpha =(x^{p^s}-\gamma_0)(x^{2p^s}+\gamma_0x^{p^s}+\gamma_0^2)$. By the Chinese Remainder Theorem, 
    \begin{equation*}
        \frac{\mathcal{R}[x]}{\langle x^{3p^s}-\alpha \rangle} \cong \frac{\mathcal{R}[x]}{\langle x^{p^s}-\gamma_0  \rangle}\oplus\frac{\mathcal{R}[x]}{\langle x^{2p^s}+\gamma_0x^{p^s}+\gamma_0^2\rangle}.
    \end{equation*}

\noindent This means that a direct sum of a $\gamma_0$-constacyclic code of length $p^s$ over $\mathcal{R}$ and ideals of $\frac{\mathcal{R}[x]}{\langle x^{2p^s}+\gamma_0x^{p^s}+\gamma_0^2\rangle}$ can be used to represent each $\alpha$-constacyclic code of length $3p^s$.
    
    %In each of the next three sections, we consider one of the subcases when $\alpha$ is non-cube in $\mathcal{R}$.

%\section{ Suppose $\alpha=\alpha_1 + \alpha_2 u + \alpha_3 v + \alpha_4 uv$ is not a square in $\mathcal{R}$, where $\alpha_1 , \alpha_2, \alpha_3 \in \mathbb{F}^*_{p^m}$ and $\alpha_4 \in \mathbb{F}_{p^m}$}

\section{$\alpha=\alpha_1 + \alpha_3 v + \alpha_4 uv $ is non-cube in $\mathcal{R}$}
 %where $\alpha_1 , \alpha_3, \alpha_4 \in \mathbb{F}^*_{p^m}$ }
    Let $\alpha=\alpha_1 + \alpha_3 v + \alpha_4 uv $ be a non-cube in  $\mathcal{R}$ where $\alpha_1 , \alpha_3, \alpha_4 \in \mathbb{F}^*_{p^m}$. Since $\alpha_1 \neq 0$, $\alpha_1 + \alpha_3 v + \alpha_4 uv $ is a unit in $\mathcal{R}$ and $(\alpha_1 + \alpha_3 v + \alpha_4 uv)$-constacyclic codes of length $3p^s$ over $\mathcal{R}$ are ideals of the quotient ring $R_{\alpha_1,  \alpha_3,  \alpha_4} =\frac{\mathcal{R}[x]}{\langle x^{3p^s}-(\alpha_1 + \alpha_3 v + \alpha_4 uv) \rangle}$.

	\begin{proposition}\label{5}
		$\alpha_1 + \alpha_3 v + \alpha_4 uv $ is non-cube in $\mathcal{R}$ if and only if  $\alpha_1$ is non-cube in $\mathbb{F}_{p^m}$.
	\end{proposition}
 %\hl{As before, this proof is not full.} \colorbox{green}{Added}

    \begin{proof}
        Suppose $\alpha=\alpha_1 + \alpha_2 u + \alpha_3 v + \alpha_4 uv$ is non-cube in $\mathcal{R}$ and assume $\alpha_1= \alpha_1^{\prime^{3}}\in \mathbb{F}_{p^m}$. Then $\alpha_1 + \alpha_2 u + \alpha_3 v + \alpha_4 uv=(\alpha_1^{\prime} +\alpha_2^{\prime}u+ \alpha_3^{\prime} v + \alpha_4^{\prime} uv)^3$, where $\alpha_2^{\prime}=0,  \alpha_3^{\prime}=3^{-1} \alpha_1^{\prime^{-2}}\alpha_3$ and $\alpha_4^{\prime}=3^{-1} \alpha_1^{\prime^{-2}}\alpha_4$,  a contradiction. 
        
        Conversely, 
        %\hlc[cyan]{we have to prove if  $\alpha_1$ is not a square in $\mathbb{F}_{p^m}$ then $\alpha$ is not a square in $\mathcal{R}$. Now consider its contrapositive. i.e., if $\alpha$ is a square in $\mathcal{R}$ then $\alpha_1$ is a square in $\mathbb{F}_{p^m}$}.
        assume that $\alpha$ is cube in $\mathcal{R}$. Then there exists $\alpha_1^{\prime} +\alpha_2^{\prime}u+ \alpha_3^{\prime} v + \alpha_4^{\prime} uv$, where $\alpha_1^{\prime}, \alpha_2^{\prime}, \alpha_3^{\prime}, \alpha_4^{\prime} \in \mathbb{F}_{p^m}$, such that 
        \begin{align*}
            \alpha=&(\alpha_1^{\prime} +\alpha_2^{\prime}u+ \alpha_3^{\prime} v + \alpha_4^{\prime} uv)^3\\
            =&\alpha_1^{\prime^{3}}+(3\alpha_1^{\prime^{2}}\alpha_2^{\prime})u+(3\alpha_1^{\prime^{2}}\alpha_3^{\prime})v+(3\alpha_1^{\prime^{2}}\alpha_4^{\prime}+6\alpha_1^{\prime}\alpha_2^{\prime}\alpha_3^{\prime})uv.
        \end{align*}
        Thus, $\alpha_1 + \alpha_3 v + \alpha_4 uv=\alpha_1^{\prime^{3}}+(3\alpha_1^{\prime^{2}}\alpha_2^{\prime})u+(3\alpha_1^{\prime^{2}}\alpha_3^{\prime})v+(3\alpha_1^{\prime^{2}}\alpha_4^{\prime}+6\alpha_1^{\prime}\alpha_2^{\prime}\alpha_3^{\prime})uv$. By comparing coefficients, we have $\alpha_1= \alpha_1^{\prime^{3}} \in \mathbb{F}_{p^m}$.
    \end{proof}
\begin{remark}
   We have $\alpha_1={\alpha_0}^{p^{s}}$. By Lemma \ref{5}, $\alpha$ is non-cube in $\mathcal{R}$ if and only if $\alpha_0$ is non-cube in $\mathbb{F}_{p^m}$.
\end{remark}
	\begin{proposition}
		Any non zero polynomial $c_1x^2+c_2x+c_3 \in \mathbb{F}_{p^m}[x]$ is a unit in $R_{\alpha_1, \alpha_3, \alpha_4}$.
	\end{proposition}
\begin{proof}
       Consider a non zero polynomial $g(x)=c_1x^2+c_2x+c_3$ in $\mathbb{F}_{p^m}[x]$ i.e., $c_1,c_2,c_3 \in \mathbb{F}_{p^m}$ are not all zeros. 
        \begin{itemize}
            \item If $deg(g(x))=0$. i.e., $c_1=c_2=0, c_3\neq 0$ and $g(x) =c_1x+c_2.$ Thus, $g(x)=c_3$ is a unit.
            \item If $deg(g(x))=1$. i.e., $c_1=0, c_2\neq 0$ and $g(x)=c_2x+c_3$.
		%i.e. $a, b\in \mathbb{F}_{p^m}$ such that not both zeros.
            In $R_{\alpha_1, \alpha_3, \alpha_4}$ we have 
            \begin{align*}
                (x-c_3)^{p^{s}}(x^2+c_3x+c_3^2)^{p^{s}}&=(x^3-c_3^2)^{p^{s}}\\
                &=x^{3p^{s}}-c_3^{2p^{s}}\\
                 &=\alpha_1+ \alpha_3 v + \alpha_4 uv -c_3^{3p^{s}}\\
                 &=(\alpha_1-c_3^{3p^{s}})+( \alpha_3v+\alpha_4 uv).
            \end{align*} 
        Not that $\alpha_1-c_2^{3p^{s}}$ is unit in $\mathbb{F}_{p^m}$ as $\alpha_1$ is non-cube in $\mathbb{F}_{p^m}$. Hence, $(\alpha_1-c_2^{3p^{s}})+( \alpha_3v+\alpha_4 uv)$ is a unit in $\mathcal{R}$. Therefore, 
        \begin{center}
            $(x-c_3)^{-1}=(x-c_3)^{p^{s-1}}(x^2+c_3x+c_3^2)^{p^{s}}(\alpha_1-c_3^{3p^{s}}+ \alpha_3v+\alpha_4 uv)^{-1}$.
        \end{center}
        For any $c_2\neq 0$ in $\mathbb{F}_{p^m}$,  we have 
        \begin{align*}
            (c_2x+c_3)^{-1}&=c_2^{-1}(x-c_2^{-1}(-c_3))^{-1}\\
            &=c_2^{-1}(x-c_2^{-1}(-c_3))^{p^{s-1}}(x^2+c_2^{-1}(-c_3)x+(c_2^{-1}(-c_3))^2)^{p^{s}}(\alpha_1-c_1^{-3p^{s}}c_2^{3p^{s}}+ \alpha_3v+\alpha_4 uv)^{-1}.
        \end{align*}

        \item If $deg(g(x))=2$. i.e., $c_1\neq 0$ and $g(x)=c_1x^2+c_2x+c_3$. In $R_{\alpha_1,  \alpha_3, \alpha_4}$ we have 
    \begin{align*}
        g^{-1}(x)=&(c_1x^2+c_2x+c_3)^{-1}\\
        =&c^{-1}_1(x^2+c^{-1}_1c_2x+c^{-1}_1c_3)^{-1}\\
        =&c^{-1}_1(x^2+c^{-1}_1c_2x+c^{-1}_1c_3)^{p^s-1}(x^2+c^{-1}_1c_2x+c^{-1}_1c_3)^{-p^s}(x-c^{-1}_1c_2)^{-p^s}(x-c^{-1}_1c_2)^{p^s}\\
        =&c^{-1}_1(x^2+c^{-1}_1c_2x+c^{-1}_1c_3)^{p^s-1}(x-c^{-1}_1c_2)^{p^s}\big[(x^2+c^{-1}_1c_2x+c^{-1}_1c_3)(x-c^{-1}_1c_2)\big]^{-p^s}\\
        =&c^{-1}_1(x^2+c^{-1}_1c_2x+c^{-1}_1c_3)^{p^s-1}(x-c^{-1}_1c_2)^{p^s}\big[x^3+(c^{-1}_1c_3-c^{-2}_1c^{2}_2)x-c^{-2}_1c_2c_3\big]^{-p^s}\\
        =&c^{-1}_1(x^2+c^{-1}_1c_2x+c^{-1}_1c_3)^{p^s-1}(x-c^{-1}_1c_2)^{p^s}\big[x^{3p^s}+(c^{-1}_1c_3-c^{-2}_1c^{2}_2)^{p^s}x^{p^s}-(c^{-2}_1c_2c_3)^{p^s}\big]^{-1}\\
        =&c^{-1}_1(x^2+c^{-1}_1c_2x+c^{-1}_1c_3)^{p^s-1}(x-c^{-1}_1c_2)^{p^s} \big[ \alpha_1 + \alpha_3 v + \alpha_4 uv+(c^{-1}_1c_3-c^{-2}_1c^{2}_2)^{p^s}x^{p^s}-(c^{-2}_1c_2c_3)^{p^s}\big]^{-1}\\
        =&c^{-1}_1(x^2+c^{-1}_1c_2x+c^{-1}_1c_3)^{p^s-1}(x-c^{-1}_1c_2)^{p^s}\big[ (\alpha_0-c^{-2}_1c_2c_3 +(c^{-1}_1c_3-c^{-2}_1c^{2}_2)x)^{p^s} + \alpha_3 v + \alpha_4 uv\big]^{-1}
    \end{align*}
So, $g(x)$ is invertible if and only if $\alpha_0-c^{-2}_1c_2c_3 +(c^{-1}_1c_3-c^{-2}_1c^{2}_2)x$ is invertible. Suppose $\alpha_0-c^{-2}_1c_2c_3 +(c^{-1}_1c_3-c^{-2}_1c^{2}_2)x=0$. Then $\alpha_0-c^{-2}_1c_2c_3=0$ and  $c^{-1}_1c_3-c^{-2}_1c^{2}_2=0.$ This implies that $\alpha_0=(c_1^{-1}c_2)^{3}$, a contradiction as $\alpha_0$ is non-cube. 
 \end{itemize}
 Therefore, any non zero polynomial $c_1x^2+c_2x+c_3 \in \mathbb{F}_{p^m}[x]$ is a unit in $R_{\alpha_1,\alpha_3, \alpha_4}$.
       
    \end{proof}

	\begin{lemma}\label{8.1}
		In $\mathcal{R}$,  $\langle v \rangle= \langle (x^3-\alpha_0)^{p^{s}} \rangle$, $x^3-\alpha_0$ is a nilpotent element with nilpotency index $2p^s$.
	\end{lemma}
	\begin{proof}
		Let $x^3-\alpha_0 \in \mathcal{R}$. As $p$ is the characteristic and $p\vert\binom{p^s}{i}$,  for every $i$, $1\leq i\leq p^s-1$,  we have 
		\begin{align*}
			(x^3-\alpha_0)^{p^{s}}&=x^{3p^{s}}-\alpha_0^{p^{s}}+\sum_{i=1}^{p^s-1}(-1)^{p^{s}-i}\binom{p^s}{i}(x^{3})^{i}(\alpha_0)^{p^{s}-i}\\
			&=x^{3p^{s}}-\alpha_0^{p^{s}}\\
			&=x^{3p^{s}}-\alpha_1\\
			&=\alpha_3 v+\alpha_4 uv\\
    &=v(\alpha_3+\alpha_4 u).
		\end{align*}
		Since $\alpha_3 \neq 0$,  $(\alpha_3+\alpha_4 u)$ is a unit in $R_{\alpha_1,\alpha_3, \alpha_4}$.
		Thus, $\langle v \rangle =\langle (x^3-\alpha_0)^{p^{s}} \rangle$. Hence,  the nilpotency index of $(x^3-\alpha_0)$ is $2p^s$.
	\end{proof}

	\begin{lemma}
		The ring $R_{\alpha_1, \alpha_3, \alpha_4}$ is a local ring with the unique maximal ideal  $\langle (x^3-\alpha_0), u\rangle $, but $R_{\alpha_1, \alpha_3, \alpha_4}$ is not a chain ring.
	\end{lemma}
	\begin{proof}
		Let $h(x)$ be a polynomial in $R_{\alpha_1, \alpha_3, \alpha_4}$.  Then it can be expressed as  $h(x)=h_1(x)+uh_2(x)+vh_3(x)+uvh_4(x)$,  where $h_i(x)$, $1\leq i\leq 4$, are polynomials of degree less than or equal to $3p^s-1$ over $\mathbb{F}_{p^m}$. Thus, $h(x)$ can be uniquely expressed as
		\begin{align*}
			h(x)=&\sum_{\ell=0}^{p^s-1}(a_{0\ell}x^{2}+b_{0\ell}x+c_{0\ell})(x^3-\alpha_0)^\ell+u\sum_{\ell=0}^{p^s-1}(a_{1\ell}x^{2}+b_{1\ell}x+c_{1\ell})(x^3-\alpha_0)^\ell\\
			&+v\sum_{\ell=0}^{p^s-1}(a_{2\ell}x^{2}+b_{2\ell}x+c_{2\ell})(x^3-\alpha_0)^\ell+uv\sum_{\ell=0}^{p^s-1}(a_{3\ell}x^{2}+b_{3\ell}x+c_{3\ell})(x^3-\alpha_0)^\ell\\
			=&(a_{00}x^2+b_{00}x+c_{00})+\sum_{\ell=1}^{p^s-1}(a_{0\ell}x^{2}+b_{0\ell}x+c_{0\ell})(x^3-\alpha_0)^\ell+u\sum_{\ell=0}^{p^s-1}(a_{1\ell}x^2+b_{1\ell}x+c_{1\ell})(x^3-\alpha_0)^\ell\\
			&+v\begin{pmatrix}
				\sum\limits_{\ell=0}^{p^s-1}(a_{2\ell}x^{2}+b_{2\ell}x+c_{2\ell})(x^3-\alpha_0)^\ell+u\sum\limits_{\ell=0}^{p^s-1}(a_{3\ell}x^{2}+b_{3\ell}x+c_{3\ell})(x^3-\alpha_0)^\ell
			\end{pmatrix},
		\end{align*} 
		where $a_{j\ell},  b_{j\ell}, c_{j\ell} \in \mathbb{F}_{p^m}$,  $j=0, 1, 2, 3$. By Lemma \ref{8.1},  we have $v=(\alpha_3+\alpha_4 u)^{-1}(x^3-\alpha_0)^{p^{s}}$. Thus, $h(x)$ can be written as
		\begin{align}\label{eqn 5.1}
			\notag h(x)=&(a_{00}x^2+b_{00}x+c_{00})+(x^3-\alpha_0)\sum_{\ell=1}^{p^s-1}(a_{0\ell}x^{2}+b_{0\ell}x+c_{0\ell})(x^3-\alpha_0)^{\ell-1}+u\sum_{\ell=0}^{p^s-1}(a_{1\ell}x^2+b_{1\ell}x+c_{1\ell})(x^3-\alpha_0)^\ell\\
			\notag	&+(\alpha_3+\alpha_4 u)^{-1}(x^3-\alpha_0)^{p^{s}} \begin{pmatrix}
				\sum\limits_{\ell=0}^{p^s-1}(a_{2\ell}x^{2}+b_{2\ell}x+c_{2\ell})(x^3-\alpha_0)^\ell+u\sum\limits_{\ell=0}^{p^s-1}(a_{3\ell}x^{2}+b_{3\ell}x+c_{3\ell})(x^3-\alpha_0)^\ell
			\end{pmatrix}\\
			=&(a_{00}x^2+b_{00}x+c_{00})+(x^3-\alpha_0)\sum\limits_{\ell=1}^{2p^s-1}(a^{\prime}_{0\ell}x^2+b^{\prime}_{0\ell}x+c^{\prime}_{0\ell})(x^3-\alpha_0)^{\ell-1}+u\sum\limits_{\ell=0}^{2p^s-1}(a^{\prime}_{1\ell}x^2+b^{\prime}_{1\ell}x+c^{\prime}_{1\ell})(x^3-\alpha_0)^\ell.
		\end{align} 
		As both $u$ and $(x^3-\alpha_0)$ are nilpotent  elements in $R_{\alpha_1, \alpha_3, \alpha_4}$,  the polynomial $h(x)$ is non-unit if and only if $a_{00}=b_{00}=c_{00}=0$. Hence $\langle (x^3-\alpha_0),  u \rangle $ represents the set of all non-units of $R_{\alpha_1, \alpha_3, \alpha_4}$,  which shows that $R_{\alpha_1,  \alpha_3, \alpha_4}$ is a local ring, and the unique maximal ideal is given by $\langle (x^3-\alpha_0), u \rangle $. It is easy to verify that $x^3-\alpha_0 \notin \langle u \rangle $ and $u\notin \langle (x^3-\alpha_0) \rangle $. Therefore,  the maximal ideal $\langle (x^3-\alpha_0),  u \rangle $ is not principal. By Proposition \ref{prop2.1},  $R_{\alpha_1, \alpha_3, \alpha_4}$ is not a chain ring.
	\end{proof}
%\begin{remark}
   % The structure of $(\alpha_1 + \alpha_3 v + \alpha_4 uv)$-constacyclic codes is obtained in a similar fashion as in the Theorem \ref{4.6}.
%\end{remark}

The ideals of the ring $R_{\alpha_1, \alpha_3, \alpha_4}$ are characterized by the following theorem.  Hence, it accounts for all $(\alpha_1 + \alpha_3 v + \alpha_4 uv)$-constacyclic codes of length $3p^s$ over $\mathcal{R}$.
	\begin{theorem}\label{4.6}
		Distinct ideals of the ring $R_{\alpha_1,  \alpha_3, \alpha_4}$ are
		\begin{itemize}
			\item Type A: Trivial ideals
			\begin{center}
				$\langle 0 \rangle $,  $\langle 1 \rangle $.
			\end{center}
			\item Type B: Principal ideals with nonmonic generators
			\begin{center}
				$\langle u(x^3-\alpha_0)^\ell \rangle $,
			\end{center}
			where $0\leq  \ell \leq 2p^s-1$.
			\item Type C: Principal ideals with monic generators
			\begin{center}
				$\langle (x^3-\alpha_0)^\ell +u(x^3-\alpha_0)^t z(x) \rangle $,
			\end{center}
			where $1\leq  \ell \leq 2p^s-1$,  $0\leq  t < \ell $ and either $z(x)$ is 0 or $z(x)$ is an invertible which can be represented as $z(x)=\sum\limits_{\kappa}^{}(z_{0\kappa}x^2+z_{1\kappa}x+z_{2\kappa})(x^3-\alpha_0)^{\kappa}$, where $z_{0\kappa},  z_{1\kappa}, z_{2\kappa} \in \mathbb{F}_{p^m}$ and $z_{00}x^2+z_{10}x+ z_{20}\neq 0$.
			\item Type D: Non-principal ideals
			\begin{center}
				$\langle (x^3-\alpha_0)^\ell +u\sum\limits_{\kappa=0}^{\mu-1}(a_{\kappa}x^2+b_{\kappa}x+c_{\kappa})(x^3-\alpha_0)^{\kappa}, u (x^3-\alpha_0)^{\mu} \rangle $,
			\end{center}
			where $0\leq  \ell \leq 2p^s-1$,  $a_{\kappa}, b_{\kappa},c_{\kappa} \in  \mathbb{F}_{p^m}$ and $\mu\ < \Im$ with $\Im$ being the smallest integer such that 
			\begin{center}
				$u (x^3-\alpha_0)^{\Im} \in \langle (x^3-\alpha_0)^\ell +u\sum\limits_{\kappa=0}^{\mu-1}(a_{\kappa}x^2+b_{\kappa}x+c_{\kappa})(x^3-\alpha_0)^{\kappa}  \rangle $
			\end{center}
			or equivalently, 
			\begin{center}
				$\langle (x^3-\alpha_0)^\ell +u(x^3-\alpha_0)^t z(x),  u (x^3-\alpha_0)^{\mu} \rangle $
			\end{center}
			with $z(x)$ as in Type C and deg $(z(x))\leq\mu-t-1$.
		\end{itemize}
	\end{theorem}
	\begin{proof}
		$\langle 0 \rangle $,  $\langle 1 \rangle $ are trivial ideals. By Equation \ref{eqn 4.1},  any element $h(x)$ of the ring $R_{\alpha_1, \alpha_2,  \alpha_3, \alpha_4}$ can be expressed as
		\begin{center}
			$h(x)=\sum\limits_{\ell=0}^{2p^s-1}(a^{\prime}_{0\ell}x^2+b^{\prime}_{0\ell}x+c^{\prime}_{0\ell})(x^3-\alpha_0)^{\ell}+u\sum\limits_{\ell=0}^{2p^s-1}(a^{\prime}_{1\ell}x^2+b^{\prime}_{1\ell}x+c^{\prime}_{1\ell})(x^3-\alpha_0)^\ell.$
		\end{center}
		Let $I\subseteq R_{\alpha_1, \alpha_3, \alpha_4}$ be an arbitrary non-trivial ideal. Then we have the following possibilities for $I$:
		\begin{itemize}
			\item \textbf{Case 1}:  $I\subseteq \langle u \rangle$. Then, every element of $I$ has the form $u\sum\limits_{\ell=0}^{2p^s-1}(a_{1\ell}x^2+b_{1\ell}x+c_{1\ell})(x^3-\alpha_0)^\ell$ where $a_{1\ell}, b_{1\ell}, c_{1\ell}\in \mathbb{F}_{p^m}$. Then there exists an element $a \in I$ with the smallest integer $ \partial$ such that $a_{1 \partial}x^2+b_{1 \partial}x+c_{1 \partial} \neq 0$. Hence each element $c(x) \in  I$ has the form
			\begin{equation*}
				c(x)=u(x^3-\alpha_0)^ \partial\sum\limits_{\ell= \partial}^{2p^s-1}(\Tilde{a}_{1\ell}x^2+\Tilde{b}_{1\ell}x+\Tilde{c}_{1\ell})(x^3-\alpha_0)^{\ell- \partial}
			\end{equation*}
			which implies $I \subseteq \langle u(x^3-\alpha_0)^ \partial \rangle$.
			Furthermore, we have $a \in I$ with
			\begin{align*}
				a&=u(x^3-\alpha_0)^ \partial\sum\limits_{\ell= \partial}^{2p^s-1}(a_{1\ell}x^2+b_{1\ell}x+c_{1\ell})(x^3-\alpha_0)^{\ell- \partial}\\
				&=u(x^3-\alpha_0)^ \partial
				\begin{pmatrix}
					a_{1 \partial}x^2+b_{1\partial}x+c_{1 \partial}+\sum\limits_{\ell= \partial+1}^{2p^s-1}(a_{1\ell}x^2+b_{1\ell}x+c_{1\ell})(x^3-\alpha_0)^{\ell- \partial}
				\end{pmatrix} 
			\end{align*}
			Since $a_{1 \partial}x^2+b_{1\partial}x+c_{1 \partial} \neq 0$,  $(a_{1 \partial}x^2+b_{1\partial}x+c_{1 \partial})+\sum\limits_{\ell= \partial+1}^{2p^s-1}(a_{1\ell}x^2+b_{1\ell}x+c_{1\ell})(x^3-\alpha_0)^{\ell- \partial}$ is a unit.  Therefore, it implies that $u(x^3-\alpha_0)^ \partial \in I$. Thus, $I = \langle u(x^3-\alpha_0)^ \partial \rangle$,  $0\leq   \partial \leq 2p^s-1$, and  it is an ideal of Type B.

			\item \textbf{Case 2}: $I\not\subseteq \langle u \rangle$. Let us represent $I_u$ as the set of all elements of $I$ modulo $u$. Thus $I_u$ is an ideal of $\frac{(\mathbb{F}_{p^m} + v\mathbb{F}_{p^m})[x]}{\langle x^{3p^s}-(\alpha_1 + \alpha_3 v) \rangle}$. From \cite{chen2016constacyclic}, any ideal of  $\frac{(\mathbb{F}_{p^m} + v\mathbb{F}_{p^m})[x]}{\langle x^{3p^s}-(\alpha_1 + \alpha_3 v) \rangle}$ is of the form $\langle (x^3-\alpha_0)^{\kappa} \rangle $,  where  $0\leq  \kappa \leq 2p^s$. Therefore,  $I_u=\langle (x^3-\alpha_0)^\ell \rangle $ for $0\leq \ell  \leq 2p^s-1$. As a result,  there exists an element of the form
			\begin{equation*}
				c(x)=\sum\limits_{\kappa=0}^{2p^s-1}(a_{0\kappa}x^2+b_{0\kappa}x+c_{0\kappa})(x^3-\alpha_0)^{\kappa}+u\sum\limits_{\kappa=0}^{2p^s-1}(a_{1\kappa}x^2+b_{1\kappa}x+c_{1\kappa})(x^3-\alpha_0)^{\kappa} \in R_{\alpha_1, \alpha_3, \alpha_4} ,
			\end{equation*}
			where $a_{0\kappa},  b_{0\kappa}, c_{0\kappa} a_{1\kappa},  b_{1\kappa},c_{1\kappa} \in \mathbb{F}_{p^m}$ such that $(x^3-\alpha_0)^\ell+uc(x) \in I$. Since
			\begin{equation*}
				(x^3-\alpha_0)^\ell+uc(x)=(x^3-\alpha_0)^\ell+u\sum\limits_{\kappa=0}^{2p^s-1}(a_{0\kappa}x^2+b_{0\kappa}x+c_{0\kappa})(x^3-\alpha_0)^{\kappa} \in I
			\end{equation*}
			and
			\begin{equation*}
				u(x^3-\alpha_0)^ \partial=u\begin{bmatrix}
					(x^3-\alpha_0)^\ell+uc(x)
				\end{bmatrix}
				(x^3-\alpha_0)^{ \partial-\ell} \in I
			\end{equation*}
			with $\ell\leq  \partial \leq 2p^s-1$,  it follows that 
			\begin{equation*}
				(x^3-\alpha_0)^\ell+u\sum\limits_{\kappa=0}^{\ell-1}(a_{0\kappa}x^2+b_{0\kappa}x+c_{0\kappa})(x^3-\alpha_0)^{\kappa} \in I.
			\end{equation*}
			Now consider two subcases.
			\begin{itemize}
				\item \textbf{Subcase 2a:} $I=\langle (x^3-\alpha_0)^\ell+u\sum\limits_{\kappa=0}^{\ell-1}(a_{0\kappa}x^2+b_{0\kappa}x+c_{0\kappa})(x^3-\alpha_0)^{\kappa} \rangle$. Then $I$ can be represented as 
				\begin{equation*}
					I=\langle (x^3-\alpha_0)^\ell+u(x^3-\alpha_0)^{t} z(x) \rangle,
				\end{equation*}
				where $z(x)$ is 0 or $z(x)$ is an invertible which can be represented as $z(x)=\sum\limits_{\kappa}^{}(z_{0\kappa}x^2+z_{1k}x+z_{2\kappa})(x^3-\alpha_0)^{\kappa} $ where $z_{0\kappa},  z_{1\kappa}, z_{2\kappa} \in \mathbb{F}_{p^m}$ and $z_{00}x^2+z_{10}x+z_{20} \neq 0$ and $0 \leq t<\ell$. Therefore, $I$ is of Type C.
				\item \textbf{Subcase 2b:} $\langle (x^3-\alpha_0)^\ell+u\sum\limits_{\kappa=0}^{\ell-1}(a_{0\kappa}x^2+b_{0\kappa}x+c_{0\kappa})(x^3-\alpha_0)^{\kappa} \rangle \subsetneq I$. Then there exists 
				\begin{equation*}
					g(x) \in I \backslash \langle (x^3-\alpha_0)^\ell+u\sum\limits_{\kappa=0}^{\ell-1}(a_{0\kappa}x^2+b_{0\kappa}x+c_{0\kappa})(x^3-\alpha_0)^{\kappa} \rangle.
				\end{equation*}
				Therefore, there exists a polynomial $h(x) \in  R_{\alpha_1, \alpha_3, \alpha_4} $ such that
				\begin{equation*}
					0 \neq z(x)=g(x)-h(x)\begin{bmatrix}
						(x^3-\alpha_0)^\ell+u\sum\limits_{\kappa=0}^{\ell-1}(a_{0\kappa}x^2+b_{0\kappa}x+c_{0\kappa})(x^3-\alpha_0)^{\kappa}
					\end{bmatrix} \in I
				\end{equation*}
				and thus $z(x)$ can be represented as 
				\begin{equation*}
					z(x)=\sum\limits_{\kappa=0}^{\ell-1}(z_{0\kappa}x^2+z^{\prime}_{0\kappa}x+z^{\prime \prime}_{0\kappa})(x^3-\alpha_0)^{\kappa}+u\sum\limits_{\kappa=0}^{\ell-1}(z_{1\kappa}x^2+z^{\prime}_{1\kappa}x+z^{\prime \prime}_{1\kappa})(x^3-\alpha_0)^{\kappa},
				\end{equation*}
				where $z_{0\kappa},  z^{\prime}_{0\kappa},  z^{\prime \prime}_{0\kappa}, z_{0\kappa}  z^{\prime}_{1\kappa}, z^{\prime \prime}_{1\kappa} \in \mathbb{F}_{p^m}$. Now $z(x)$ reduced modulo $u$ is contained in $I_u=\langle (x^3-\alpha_0)^\ell \rangle $ and thus $z_{0\kappa}=  z^{\prime}_{0\kappa}=z^{\prime \prime}_{0\kappa}=0$ for all $\kappa$, $0\leq \kappa \leq \ell-1$. i.e., $z(x)=u\sum\limits_{\kappa=0}^{\ell-1}(z_{1\kappa}x^2+z^{\prime}_{1\kappa}x+z^{\prime \prime}_{1\kappa})(x^3-\alpha_0)^{\kappa}$. Since $z(x) \neq 0$,  there exists a smallest integer $\zeta$,  $0 \leq \zeta \leq \ell-1$,  such that $z_{1 \zeta}x^2+z^{\prime}_{1 \zeta}x+z^{\prime}_{1 \zeta} \neq 0$. Then
				\begin{align*}
					z(x)&=u\sum\limits_{\kappa=0}^{\ell-1}(z_{1\kappa}x^2+z^{\prime}_{1\kappa}x+z^{\prime \prime}_{1\kappa})(x^3-\alpha_0)^{\kappa}\\
					&=u(x^3-\alpha_0)^{\zeta} \Biggr[
						z_{1 \zeta}x^2+z^{\prime}_{1 \zeta} x+z^{\prime \prime}_{1 \zeta} +\sum\limits_{\kappa=\zeta+1}^{\ell-1}(z_{1\kappa}x^2+z^{\prime}_{1\kappa}x+z^{\prime \prime}_{1\kappa})(x^3-\alpha_0)^{\kappa-\zeta}
					\Biggr].
				\end{align*}
				Since $z_{1 \zeta}x^2+z^{\prime}_{1 \zeta} x+z^{\prime \prime}_{1 \zeta} \neq 0$ and $z_{1 \zeta}x^2+z^{\prime}_{1 \zeta} x+z^{\prime \prime}_{1 \zeta} +\sum\limits_{\kappa=\zeta+1}^{\ell-1}(z_{1\kappa}x^2+z^{\prime}_{1\kappa}x+z^{\prime \prime}_{1\kappa})(x^3-\alpha_0)^{\kappa-\zeta}$ is a unit  in $R_{\alpha_1,  \alpha_3, \alpha_4} $,  we have 
				\begin{align*}
					u(x^3-\alpha_0)^{\zeta} & =\begin{pmatrix}
						z_{1 \zeta}x^2+z^{\prime}_{1 \zeta} x+z^{\prime \prime}_{1 \zeta} +\sum\limits_{\kappa=\zeta+1}^{\ell-1}(z_{1\kappa}x^2+z^{\prime}_{1\kappa}x+z^{\prime \prime}_{1\kappa})(x^3-\alpha_0)^{\kappa-\zeta}
					\end{pmatrix}^{-1} z(x) \in I.
				\end{align*}
			For any 
				\begin{equation*}
					g(x) \in I \backslash \langle (x^3-\alpha_0)^\ell+u\sum\limits_{\kappa=0}^{\ell-1}(a_{0\kappa}x^2+b_{0\kappa}x+c_{0\kappa})(x^3-\alpha_0)^{\kappa} \rangle
				\end{equation*}
					we have proved that there is an integer $\zeta$, $0\leq \zeta \leq \ell-1 $ such that $u(x^3-\alpha_0)^\zeta \in I$. Let 
				\begin{equation*}
					\mu=\min\{\quad \zeta \quad : \quad g(x) \in I \backslash \langle (x^3-\alpha_0)^\ell+u\sum\limits_{\kappa=0}^{\ell-1}(a_{0\kappa}x^2+b_{0\kappa}x+c_{0\kappa})(x^3-\alpha_0)^{\kappa} \rangle \}.
				\end{equation*}
				Then
				\begin{equation*}
					\langle (x^3-\alpha_0)^\ell+u\sum\limits_{\kappa=0}^{\ell-1}(a_{0\kappa}x^2+b_{0\kappa}x+c_{0\kappa})(x^3-\alpha_0)^{\kappa},  u (x^3-\alpha_0)^{\mu}\rangle \subseteq I.
				\end{equation*}
				Furthermore, according to the preceding construction, for any element $g(x)$ of $I$ there exists a polynomial $h(x) \in I$ such that
				\begin{equation*}
					g(x)-h(x)\begin{bmatrix}
						(x^3-\alpha_0)^\ell+u\sum\limits_{\kappa=0}^{\ell-1}(a_{0\kappa}x^2+b_{0\kappa}x+c_{0\kappa})(x^3-\alpha_0)^{\kappa}
					\end{bmatrix} \in \langle u (x^3-\alpha_0)^{\mu} \rangle,
				\end{equation*}
				which implies that 
				\begin{equation*}
					g(x)\in \langle (x^3-\alpha_0)^\ell+u\sum\limits_{\kappa=0}^{\ell-1}(a_{0\kappa}x^2+b_{0\kappa}x+c_{0\kappa})(x^3-\alpha_0)^{\kappa},  u (x^3-\alpha_0)^{\mu}\rangle.
				\end{equation*}
				Thus,
				\begin{align*}
					I&=\langle (x^3-\alpha_0)^\ell+u\sum\limits_{\kappa=0}^{\ell-1}(a_{0\kappa}x^2+b_{0\kappa}x+c_{0\kappa})(x^3-\alpha_0)^{\kappa},  u (x^3-\alpha_0)^{\mu}\rangle \\
					&=\langle (x^3-\alpha_0)^\ell+u\sum\limits_{\kappa=0}^{\mu-1}(a_{0\kappa}x^2+b_{0\kappa}x+c_{0\kappa})(x^3-\alpha_0)^{\kappa},  u (x^3-\alpha_0)^{\mu}\rangle.
				\end{align*}
				Let $\Im$ be the smallest integer such that $u (x^3-\alpha_0)^{\Im} \in \langle (x^3-\alpha_0)^\ell+u\sum\limits_{\kappa=0}^{\ell-1}(a_{0\kappa}x^2+b_{0\kappa}x+c_{0\kappa})(x^3-\alpha_0)^{\kappa}\rangle $. If $\mu \geq \Im$,  then 
				\begin{align*}
					I&=\langle (x^3-\alpha_0)^\ell+u\sum\limits_{\kappa=0}^{\mu-1}(a_{0\kappa}x^2+b_{0\kappa}x+c_{0\kappa})(x^3-\alpha_0)^{\kappa},  u (x^3-\alpha_0)^{\mu}\rangle \\
					&=\langle (x^3-\alpha_0)^\ell+u\sum\limits_{\kappa=0}^{\ell-1}(a_{0\kappa}x^2+b_{0\kappa}x+c_{0\kappa})(x^3-\alpha_0)^{\kappa}\rangle.
				\end{align*} 
				This contradicts the assumption of this case, $\mu < \Im$,  showing that $I$ is of Type D.
			\end{itemize}
		\end{itemize}
	\end{proof}
	The number $\Im$ is significant in Type D based on the categorization in Theorem \ref{4.6}. We now take a closer look at  $\Im$.

\begin{theorem}\label{8.2}
   Let $\mathcal{C}=\langle (x^3-\alpha_0)^\ell+u(x^3-\alpha_0)^{t}z(x)\rangle$ and $\Im$ be the smallest integer such that $u (x^3-\alpha_0)^{\Im} \in \mathcal{C}$.		
  %Let $\Im$ be the smallest integer such that $u (x^3-\alpha_0)^{\Im} \in\langle (x^3-\alpha_0)^\ell+u(x^3-\alpha_0)^{t}z(x)\rangle $. 
  Then
		\begin{center}
			$\Im=$
			$\begin{cases}
				\ell  & \text{if}\quad z(x)=0;\\
				min\{\ell, 2p^s +t-\ell\}  &\text{if}\quad z(x)\neq 0. \\
			\end{cases}$
		\end{center}
	\end{theorem}
	\begin{proof}
		We have $\Im \leq \ell $ because
		\begin{equation*}
			u(x^3-\alpha_0)^\ell=u[(x^3-\alpha_0)^\ell+u(x^3-\alpha_0)^t z(x)] \in \mathcal{C}.
		\end{equation*}
		Suppose $z(x)=0$,  then $\mathcal{C}=\langle (x^3-\alpha_0)^\ell \rangle $,  implies $\Im =\ell$. Let us now examine the case $z(x) \neq 0$,  i.e.,  $z(x)$ is an invertible. As $u(x^3-\alpha_0)^ {\Im} \in \langle (x^3-\alpha_0)^\ell+u(x^3-\alpha_0)^t z(x) \rangle $ there exists a polynomial $f(x)=\sum\limits_{\kappa=0}^{2p^s-1}(a_{0\kappa}x^2+b_{0\kappa}x+c_{0\kappa})(x^3-\alpha_0)^{\kappa}+u\sum\limits_{\kappa=0}^{2p^s-1}(a_{1\kappa}x^2+b_{1\kappa}x+c_{1\kappa})(x^3-\alpha_0)^{\kappa} \in R_{\alpha_1, \alpha_3, \alpha_4} $, where $a_{0\kappa},  b_{0\kappa}, c_{0\kappa}, a_{1\kappa},  b_{1\kappa}, c_{1\kappa}  \in \mathbb{F}_{p^m}$ such that $u(x^3-\alpha_0)^{\Im} =f(x)[(x^3-\alpha_0)^\ell+u(x^3-\alpha_0)^t z(x)]$. Then 
		\begin{align*}
			u(x^3-\alpha_0)^{\Im} =&
			\begin{bmatrix}
				\sum\limits_{\kappa=0}^{2p^s-1}(a_{0\kappa}x^2+b_{0\kappa}x+c_{0\kappa})(x^3-\alpha_0)^{\kappa}+u\sum\limits_{\kappa=0}^{2p^s-1}(a_{1\kappa}x^2+b_{1\kappa}x+c_{1\kappa})(x^3-\alpha_0)^{\kappa}
			\end{bmatrix}\\
			& \times [(x^3-\alpha_0)^\ell +u(x^3-\alpha_0)^t z(x)] \\
			=&(x^3-\alpha_0)^\ell \sum\limits_{\kappa=0}^{2p^s-1}(a_{0\kappa}x^2+b_{0\kappa}x+c_{0\kappa})(x^3-\alpha_0)^{\kappa}\\
			&+u(x^3-\alpha_0)^\ell \sum\limits_{\kappa=0}^{2p^s-1}(a_{1\kappa}x^2+b_{1\kappa}x+c_{1\kappa})(x^3-\alpha_0)^{\kappa}\\
			&+u(x^3-\alpha_0)^t z(x)\sum\limits_{\kappa=0}^{2p^s-1}(a_{0\kappa}x^2+b_{0\kappa}x+c_{0\kappa})(x^3-\alpha_0)^{\kappa}\\
	       =&u(x^3-\alpha_0)^\ell \sum\limits_{\kappa=0}^{2p^s-\ell-1}(a_{1\kappa}x^2+b_{1\kappa}x+c_{1\kappa})(x^3-\alpha_0)^{\kappa}\\
			&+u(x^3-\alpha_0)^{2p^s+t-\ell} z(x)\sum\limits_{\kappa=0}^{\ell-1}(a_{0\kappa}x^2+b_{0\kappa}x+c_{0\kappa})(x^3-\alpha_0)^{\kappa}.
		\end{align*}
	 Thus, $\Im \geq min\{\ell, 2p^s+t-\ell\}$. Moreover, by Lemma \ref{8.1},
		\begin{align*}
			&[(x^3-\alpha_0)^\ell +u(x^3-\alpha_0)^t z(x)](x^3-\alpha_0)^{2p^s-\ell}=u(x^3-\alpha_0)^{2p^s+t-\ell}z(x) \\
   \implies & u(x^3-\alpha_0)^{2p^s+t-\ell}=[(x^3-\alpha_0)^\ell +u(x^3-\alpha_0)^t z(x)](x^3-\alpha_0)^{2p^s-\ell}z^{-1}(x) \in \mathcal{C}.
		\end{align*}
Thus, $\Im \leq 2p^s+t-\ell$. Hence $\Im = min\{\ell, 2p^s+t-\ell\}$.
	\end{proof}
\begin{theorem}\label{8.3}
		Let $\mathcal{C}$ be a $(\alpha_1 +  \alpha_3 v + \alpha_4 uv)$-constacyclic code of length $3p^s$ over $\mathcal{R}$ and let $\eta_\mathcal{C}$ denote the number of codewords in $\mathcal{C}$. Then
\begin{enumerate}
    \item Type A: Trivial ideals
        \begin{itemize}
            \item If $\mathcal{C}=\langle 0 \rangle$,  then $Res(\mathcal{C})=\langle 0 \rangle=Tor(\mathcal{C})$ and $\eta_\mathcal{C}=1$.
            \item If $\mathcal{C}=\langle 1 \rangle$,  then $Res(\mathcal{C})=\langle 1 \rangle=Tor(\mathcal{C})$ and $\eta_\mathcal{C}=p^{12mp^{s}}$.
        \end{itemize}
     \item Type B: Principal ideals with nonmonic generators
    \begin{itemize}
        \item If $\mathcal{C}=\langle u(x^3-\alpha_0)^\ell \rangle $,  where $0\leq  \ell \leq 2p^s-1$,  then $Res(\mathcal{C})=\langle 0 \rangle$,  $Tor(\mathcal{C})=\langle (x^3-\alpha_0)^\ell \rangle $ and $\eta_\mathcal{C}=p^{3m(2p^{s}-\ell)}$.
    \end{itemize}
    \item Type C: Principal ideals with monic generators
    \begin{itemize}
        \item If $\mathcal{C}=\langle (x^3-\alpha_0)^\ell \rangle $,  where $1\leq  \ell \leq 2p^s-1$,  then $Res(\mathcal{C})=\langle (x^3-\alpha_0)^\ell \rangle $ and $Tor(\mathcal{C})=\langle (x^3-\alpha_0)^{\Im} \rangle $ and $\eta_\mathcal{C}=p^{6m(2p^s-\ell)}.$
				
        \item If $\mathcal{C}=\langle (x^3-\alpha_0)^\ell +u(x^3-\alpha_0)^t z(x) \rangle $,  where $1\leq  \ell \leq 2p^s-1$,  $0\leq  t < \ell  $ and $z(x)$ is an invertible, then  $Res(\mathcal{C})=\langle (x^3-\alpha_0)^\ell \rangle $ and $Tor(\mathcal{C})=\langle (x^3-\alpha_0)^{\Im} \rangle $, where $\Im$ is the smallest integer such that $u(x^3-\alpha_0)^{\Im} \in \mathcal{C}$,  which is given by $\Im = min\{ \ell, 2p^s+t-\ell\}$. Thus,
        \begin{center}
            $\eta_\mathcal{C}=$
            $\begin{cases}
                p^{6m(2p^s-\ell)} & \text{if}\quad 1\leq \ell\leq p^s+\frac{t}{2};\\
                p^{3m(2p^s-t)} & \text{if} \quad  p^s+ \frac{t}{2} < \ell \leq 2p^s-1.
            \end{cases}$
        \end{center} 
        \end{itemize}
    \item Type D: Non-principal ideals \\
    If $\mathcal{C}=\langle (x^3-\alpha_0)^\ell +u(x^3-\alpha_0)^t z(x),  u (x^3-\alpha_0)^{\mu} \rangle $,  where $1\leq  \ell \leq 2p^s-1$,  $0\leq  t < \ell  $ and $z(x)$ is 0 or $z(x)$ is an invertible, and deg $(z(x))\leq\mu-t-1$ and
    \begin{center}
			$\mu<\Im=$
			$\begin{cases}
				\ell  & \text{if}\quad z(x)=0;\\
				min\{\ell, 2p^s +t-\ell\}  &\text{if}\quad z(x)\neq 0. \\
			\end{cases}$
		\end{center}
    Then $Res(\mathcal{C})=\langle (x^3-\alpha_0)^\ell \rangle $ and $Tor(\mathcal{C})=\langle (x^3-\alpha_0)^{\mu} \rangle $. Thus, $\eta_\mathcal{C}=p^{3m(4p^s-\ell-\mu)}$.
\end{enumerate}
\end{theorem}

	\subsection{Duals of $(\alpha_1+  \alpha_3 v + \alpha_4 uv)$-constacyclic codes}
	\begin{lemma}{\label{8.4}}
		If $\mathcal{C}=\langle(x^3-\alpha_0)^\ell+u(x^3-\alpha_0)^tz(x),  u(x^3-\alpha_0)^\mu \rangle$,  then the smallest positive integer $\varepsilon$ such that $u(x^3-\alpha_0)^{\varepsilon} \in \mathcal{A}(\mathcal{C})$ is $2p^s-\ell$. 
	\end{lemma}
	\begin{proof}
		Suppose
		\begin{equation*}
			[(x^3-\alpha_0)^\ell+u(x^3-\alpha_0)^tz(x)]u(x^3-\alpha_0)^{\varepsilon}=0.
		\end{equation*}
		Then from Lemma \ref{8.1},  we have $\ell+\varepsilon \geq 2p^s$,  i.e.,  $\varepsilon \geq 2p^s-\ell$.
	\end{proof}
	\begin{theorem}\label{8.5}
		Let $\mathcal{C}=\langle u(x^3-\alpha_0)^\ell \rangle $ be a  $(\alpha_1 + \alpha_3 v + \alpha_4 uv)$-constacyclic code of length $3p^s$ over $\mathcal{R}$. Then $\mathcal{C}^{\perp}=\langle (x^3-\alpha_0^{-1})^{2p^s-\ell},  u \rangle$.
	\end{theorem}
	\begin{proof}
		As $\mathcal{C}\subseteq \langle u \rangle $ and $\mathcal{C} \subseteq \langle (x^3-\alpha_0)^\ell \rangle $,  we have that $\langle (x^3-\alpha_0^{-1})^{2p^s-\ell }\rangle =\langle (x^3-\alpha_0)^\ell \rangle ^{\perp} \subseteq \mathcal{C}^{\perp}$ and $\langle u \rangle =\langle u \rangle ^{\perp} \subseteq \mathcal{C}^{\perp}$. So $\langle (x^3-\alpha_0^{-1})^{2p^s-\ell },  u \rangle \subseteq \mathcal{C}^{\perp}$. We have $\vert \mathcal{C}\vert=p^{3m(2p^s-\ell)}$ and \begin{equation*}
			\vert \langle (x^3-\alpha_0^{-1})^{2p^s-\ell },  u \rangle \vert =p^{3m(2p^s+\ell)}.
		\end{equation*}
		Therefore, 
		\begin{equation*}
			\vert \mathcal{C}^{\perp} \vert=\frac{\vert \mathcal{R} \vert ^{3p^s}}{\vert \mathcal{C} \vert}=\frac{p^{12mp^s}}{p^{3m(3p^s-\ell)}}=p^{3m(2p^s+\ell)}=\vert \langle (x^3-\alpha_0^{-1})^{2p^s-\ell },  u \rangle \vert.
		\end{equation*}
		Hence,  $\mathcal{C}^{\perp}=\langle (x^3-\alpha_0^{-1})^{2p^s-\ell },  u \rangle$.
	\end{proof}
	\begin{theorem}\label{8.6}
		Let $\mathcal{C}$ be a $(\alpha_1 + \alpha_3 v + \alpha_4 uv)$-constacyclic code associate to the ideal $\mathcal{C}=\langle (x^3-\alpha_0)^\ell +u(x^3-\alpha_0)^t z(x) \rangle $,  where $z(x)$ is 0 or an invertible. Then the dual code $\mathcal{C}^\perp $ associate to the ideal $\mathcal{A}(\mathcal{C})^*$ is determined as follows:
		\begin{enumerate}
			\item If $z(x)$ is 0 and $1\leq \ell\leq p^s$,  then $\mathcal{A}(\mathcal{C})^* = \langle x^3-\alpha_0^{-1}  \rangle$.
			\item If $z(x)$ is an invertible and $1\leq \ell\leq p^s+\frac{t}{2}$,  then $\mathcal{A}(\mathcal{C})^*=\langle \chi^*_1(x) \rangle $,  where $\chi^*_3(x)=(-\alpha_0)^{2p^s-\ell}(x^3-\alpha_0^{-1})^{2p^s-\ell}-u(-\alpha_0)^{2p^s-2\ell+t}(x^3-\alpha_0^{-1})^{2p^s-2\ell+t} \sum \limits_{\kappa=0}^{\ell-t-1}(z_{2\kappa}x^2+z_{1\kappa}x+z_{0\kappa})(-\alpha_0)^{\kappa}(x^3-\alpha_0^{-1})^{\kappa} x^{3\ell-3t-3\kappa-2}$.
        \item If $z(x)$ is an invertible and $p^s+ \frac{t}{2} < \ell \leq 2p^s-1$,  then $\mathcal{A}(\mathcal{C})^*=\langle \chi^*_2(x),  (x^3-\alpha_0^{-1})^{2p^s-\ell} \rangle $,  where $\chi^*_5(x)= (-\alpha_0)^{\ell-t}(x^3-\alpha_0^{-1})^{\ell-t}-u\sum \limits_{\kappa=0}^{2p^s-\ell-1}(z_{2\kappa}x^2+z_{1\kappa}x+z_{0\kappa})(-\alpha_0)^{\kappa}(x^3-\alpha_0^{-1})^{\kappa} x^{3\ell-3t-3\kappa-2}$.
		\end{enumerate}
	\end{theorem}
	\begin{proof}
		By Lemma \ref{8.4}, $2p^s-\ell$ is the smallest positive integer $\varepsilon$ such that $u(x^3-\alpha_0)^{\varepsilon} \in \mathcal{A}(\mathcal{C})$. If $(x^3-\alpha_0)^{\rho}+u(x^3-\alpha_0)^{\xi}r(x)\in \mathcal{A}(\mathcal{C})$, where $r(x)$ is 0 or an invertible then [$(x^3-\alpha_0)^{\rho}+u(x^3-\alpha_0)^{\xi}r(x)][(x^3-\alpha_0)^{\ell}+u(x^3-\alpha_0)^{t}z(x)]=(x^3-\alpha_0)^{\rho+\ell}+u(x^3-\alpha_0)^{\rho+t}z(x)+u(x^3-\alpha_0)^{\xi +\ell}r(x)=0$.
\begin{enumerate}
    \item  The proof is obvious. 
    \item  If $z(x)$ is an invertible and $1\leq \ell\leq p^s+\frac{t}{2}$, then $\mathcal{C}=\langle (x^3-\alpha_0)^\ell+u(x^3-\alpha_0)^t z(x) \rangle $. If we choose $\rho=2p^s-\ell$,  $\xi=2p^s-2\ell+t$ and $r(x)=-z(x)$ then by Lemma \ref{8.1},  $\chi_1(x)=(x^3-\alpha_0)^{2p^s-\ell}-u(x^3-\alpha_0)^{2p^s-2\ell-t}z(x) \in \mathcal{A}(\mathcal{C})$. Therefore, $\langle \chi_1(x) \rangle \subseteq \mathcal{A}(\mathcal{C})$. Furthermore, 
    \begin{equation*}
        p^{6m\ell}=\vert \langle \chi_1(x) \rangle \vert\leq \vert \mathcal{A}(\mathcal{C}) \vert =\vert \mathcal{A}(\mathcal{C})^{*} \vert =\vert \mathcal{C}^{\perp}\vert=\frac{p^{12mp^s}}{\vert \mathcal{C} \vert}=\frac{p^{12mp^s}}{p^{6m(2p^s-\ell)}}=p^{6m\ell}.
    \end{equation*}
Thus, $\langle \chi_1(x) \rangle=\mathcal{A}(\mathcal{C})$. 
Let $z(x)=\sum\limits_{\kappa}^{}(z_{0\kappa}x^2+z_{1\kappa}x+z_{2\kappa})(x^3-\alpha_0)^{\kappa}$, where $z_{0\kappa},  z_{1\kappa}, z_{2\kappa} \in \mathbb{F}_{p^m}$ and $z_{00}x^2+z_{10}x+z_{20} \neq 0$. Here $\kappa \leq \Im-t-1=\ell-t-1$. Then 
\begin{align*}
    \chi_1(x)&=(x^3-\alpha_0)^{2p^s-\ell}-u(x^3-\alpha_0)^{2p^s-2\ell+t}z(x)\\
    &=(x^3-\alpha_0)^{2p^s-\ell}-u(x^3-\alpha_0)^{2p^s-2\ell+t}\sum\limits_{\kappa=0}^{\ell-t-1}(z_{0\kappa}x^2+z_{1\kappa}x+z_{2\kappa})(x^3-\alpha_0)^{\kappa}.
\end{align*}
Then $\mathcal{A}(\mathcal{C})^{*}=\langle \chi^*_1(x) \rangle $,  where $\chi^*_1(x)=(-\alpha_0)^{2p^s-\ell}(x^3-\alpha_0^{-1})^{2p^s-\ell}-u(-\alpha_0)^{2p^s-2\ell+t}(x^3-\alpha_0^{-1})^{2p^s-2\ell+t} \sum \limits_{\kappa=0}^{\ell-t-1}(z_{2\kappa}x^2+z_{1\kappa}x+z_{0\kappa})(-\alpha_0)^{\kappa}(x^3-\alpha_0^{-1})^{\kappa} x^{3\ell-3t-3\kappa-2}$.
\item If $z(x)$ is an invertible and $p^s+ \frac{t}{2} < \ell \leq 2p^s-1$,  then $\mathcal{C}=\langle (x^3-\alpha_0)^\ell+u(x^3-\alpha_0)^t z(x) \rangle $. As above $\chi_2(x)=(x^3-\alpha_0)^{\ell-t}-u z(x)\in \mathcal{A}(\mathcal{C})$. Therefore, $\langle \chi_2(x),  u(x^3-\alpha_0)^{2p^s-\ell} \rangle \subseteq \mathcal{A}(\mathcal{C})$. Furthermore, 
\begin{equation*}
    p^{3m(2p^s+t)}=\vert \langle \chi_2(x), u(x^3-\alpha_0)^{2p^s-\ell}  \rangle \vert\leq \vert \mathcal{A}(\mathcal{C}) \vert =\vert \mathcal{A}(\mathcal{C})^{*} \vert =\vert \mathcal{C}^{\perp}\vert=\frac{p^{12mp^s}}{\vert \mathcal{C} \vert}=\frac{p^{12mp^s}}{p^{3m(2p^s-t)}}=p^{3m(2p^s+t)}.
\end{equation*}
Thus, $\langle \chi_2(x), u(x^3-\alpha_0)^{2p^s-\ell}  \rangle=\mathcal{A}(\mathcal{C})$.
Let $z(x)=\sum\limits_{\kappa}^{}(z_{0\kappa}x^2+z_{1\kappa}x+z_{2\kappa})(x^3-\alpha_0)^{\kappa}$ where $z_{0\kappa},  z_{1\kappa}, z_{2\kappa} \in \mathbb{F}_{p^m}$ and $z_{00}x^2+z_{10}x+z_{20} \neq 0$. Here $\kappa \leq\Im-t-1=2p^s-\ell-1$. Then 
\begin{align*}
    \chi_2(x)&=(x^3-\alpha_0)^{\ell-t}-uz(x)\\
    &=(x^3-\alpha_0)^{\ell-t}-u\sum\limits_{\kappa=0}^{2p^s-\ell-1}(z_{0\kappa}x^2+z_{1\kappa}x+z_{2\kappa})(x^3-\alpha_0)^{\kappa}.
\end{align*}
Then  $ \mathcal{A}(\mathcal{C})^{*}=\langle \chi^*_2(x),  (x^3-\alpha_0^{-1})^{2p^s-\ell} \rangle $,  where $\chi^*_2(x)= (-\alpha_0)^{\ell-t}(x^3-\alpha_0^{-1})^{\ell-t}-u\sum \limits_{\kappa=0}^{2p^s-\ell-1}(z_{2\kappa}x^2+z_{1\kappa}x+z_{0\kappa})(-\alpha_0)^{\kappa}(x^3-\alpha_0^{-1})^{\kappa} x^{3\ell-3t-3\kappa-2}$.
\end{enumerate}
\end{proof}
\begin{theorem}\label{8.7}
    If $\mathcal{C}=\langle (x^3-\alpha_0)^\ell +u(x^3-\alpha_0)^t z(x),  u (x^3-\alpha_0)^{\mu} \rangle $,  where $0\leq  \ell \leq 2p^s-1$,  $0\leq  t < \ell  $ and $z(x)$ is 0 or $z(x)$ is an invertible, then  the dual code $\mathcal{C}^\perp $ associate to the ideal $\mathcal{A}(\mathcal{C})^*$ is determined as follows:
\begin{enumerate}
    \item If $z(x)$ is 0 then $ \mathcal{A}(\mathcal{C})^{*} = \langle \varphi_1^*(x), u(x^3-\alpha_0^{-1})^{2p^s-\ell }  \rangle$,  where $ \varphi_1^*(x)= (-\alpha_0)^{2p^s-\mu}(x^3-\alpha_0^{-1})^{2p^s-\mu} $.
    \item If $z(x)$ is an invertible  then $ \mathcal{A}(\mathcal{C})^{*} = \langle \varphi_2^*(x), u(x^3-\alpha_0^{-1})^{2p^s-\ell }  \rangle$,  where $ \varphi_2^*(x)= (-\alpha_0)^{2p^s-\mu}(x^3-\alpha_0^{-1})^{2p^s-\mu}-u(-\alpha_0)^{2p^s+t-\ell-\mu}(x^3-\alpha_0^{-1})^{2p^s+t-\ell-\mu}
    \sum\limits_{\kappa=0}^{\mu-t-1}(z_{2\kappa}x^2+z_{1\kappa}x+z_{0\kappa})(-\alpha_0)^{\kappa}(x^3-\alpha_0^{-1})^{\kappa} x^{3\ell-3t-3\kappa-2} $.
\end{enumerate}  
\end{theorem}
\begin{proof}
    \begin{enumerate}
        \item If $z(x)$ is 0 then $\mathcal{C}=\langle (x^3-\alpha_0)^\ell, u (x^3-\alpha_0)^{\mu}\rangle $. Clearly  $\varphi_1(x)=(x^3-\alpha_0)^{2p^s-\mu} \in \mathcal{A}(\mathcal{C})$. Therefore, $\langle \varphi_1(x),  u(x^3-\alpha_0)^{2p^s-\ell} \rangle \subseteq \mathcal{A}(\mathcal{C})$. Furthermore, 
        \begin{equation*}
        p^{3m(\ell+\mu)}=\vert \langle \varphi_1(x), u(x^3-\alpha_0)^{2p^s-\ell}  \rangle \vert\leq \vert \mathcal{A}(\mathcal{C}) \vert =\vert \mathcal{A}(\mathcal{C})^{*} \vert =\vert \mathcal{C}^{\perp}\vert=\frac{p^{12mp^s}}{\vert \mathcal{C} \vert}=\frac{p^{12mp^s}}{p^{3m(4p^s-\ell-\mu)}}=p^{3m(\ell+\mu)}.
        \end{equation*}
    Thus, $\langle \varphi_1(x), u(x^3-\alpha_0)^{2p^s-\ell}  \rangle=\mathcal{A}(\mathcal{C})$ and  $ \mathcal{A}(\mathcal{C})^{*} = \langle \varphi_1^*(x), u(x^3-\alpha_0^{-1})^{2p^s-\ell }  \rangle$,  where $ \varphi_1^*(x)= (-\alpha_0)^{2p^s-\mu}(x^3-\alpha_0^{-1})^{2p^s-\mu} $.
    \item If $z(x)$ is an invertible then $\mathcal{C}=\langle (x^3-\alpha_0)^\ell +u(x^3-\alpha_0)^t z(x),  u (x^3-\alpha_0)^{\mu} \rangle $. Clearly $\varphi_3(x)=(x^3-\alpha_0)^{2p^s-\mu}+u(x^3-\alpha_0)^{2p^s+t-\ell-\mu}z(x)\in \mathcal{A}(\mathcal{C})$. Therefore, $\langle \varphi_3(x),  u(x^3-\alpha_0)^{2p^s-\ell} \rangle \subseteq \mathcal{A}(\mathcal{C})$. Furthermore, 
\begin{equation*}
    p^{3m(\ell+\mu)}=\vert \langle \varphi_3(x), u(x^3-\alpha_0)^{2p^s-\ell}  \rangle \vert\leq \vert \mathcal{A}(\mathcal{C}) \vert =\vert \mathcal{A}(\mathcal{C})^{*} \vert =\vert \mathcal{C}^{\perp}\vert=\frac{p^{12mp^s}}{\vert \mathcal{C} \vert}=\frac{p^{12mp^s}}{p^{3m(4p^s-\ell-\mu)}}=p^{3m(\ell+\mu)}.
\end{equation*}
Thus, $\langle \varphi_3(x), u(x^3-\alpha_0)^{2p^s-\ell}  \rangle=\mathcal{A}(\mathcal{C})$. Let $z(x)=\sum\limits_{\kappa}^{}(z_{0\kappa}x^2+z_{1\kappa}x+z_{2\kappa})(x^3-\alpha_0)^{\kappa}$, where $z_{0\kappa},  z_{1\kappa}, z_{2\kappa} \in \mathbb{F}_{p^m}$ and $z_{00}x^2+z_{10}x+z_{20} \neq 0$. Here $\kappa \leq\mu-t-1$. Then 
\begin{align*}
    \varphi_3(x)&=(x^3-\alpha_0)^{2p^s-\mu}+u(x^3-\alpha_0)^{2p^s+t-\ell-\mu}z(x)\\
    &=(x^3-\alpha_0)^{2p^s-\mu}+u(x^3-\alpha_0)^{2p^s+t-\ell-\mu}\sum\limits_{\kappa=0}^{\mu-t-1}(z_{0\kappa}x^2+z_{1\kappa}x+z_{2\kappa})(x^3-\alpha_0)^{\kappa}.
\end{align*}
Thus,  $ \mathcal{A}(\mathcal{C})^{*} = \langle \varphi_3^*(x), u(x^3-\alpha_0^{-1})^{2p^s-\ell }  \rangle$,  where $ \varphi_3^*(x)= (-\alpha_0)^{2p^s-\mu}(x^3-\alpha_0^{-1})^{2p^s-\mu}-u(-\alpha_0)^{2p^s+t-\ell-\mu}(x^3-\alpha_0^{-1})^{2p^s+t-\ell-\mu}
\sum\limits_{\kappa=0}^{\mu-t-1}(z_{2\kappa}x^2+z_{1\kappa}x+z_{0\kappa})(-\alpha_0)^{\kappa}(x^3-\alpha_0^{-1})^{\kappa} x^{3\ell-3t-3\kappa-2} $.
\end{enumerate}
\end{proof}

\begin{remark}
    By interchanging the roles of $u$ and $v$ in the above section we get the structure, number of codewords, and duals of $(\alpha_1 + \alpha_2 u + \alpha_4 uv)$-constacyclic codes of length $3p^s$ over $\mathcal{R}$.
\end{remark}

\section{  $\alpha=\alpha_1 + \alpha_2 u + \alpha_3 v + \alpha_4 uv$ is non-cube in $\mathcal{R}$}
%where $\alpha_1 , \alpha_2, \alpha_3 \in \mathbb{F}^*_{p^m}$ and $\alpha_4 \in \mathbb{F}_{p^m}$}} 

    Let  $\alpha=\alpha_1 + \alpha_2 u + \alpha_3 v + \alpha_4 uv$ be a non-cube in $\mathcal{R}$, where $\alpha_1 , \alpha_2, \alpha_3 \in \mathbb{F}^*_{p^m}$ and $\alpha_4 \in \mathbb{F}_{p^m}$. Since $\alpha_1 \in\mathbb{F}^*_{p^m}$,  $\alpha=\alpha_1 + \alpha_2 u + \alpha_3 v + \alpha_4 uv$ is a unit in $\mathcal{R}$. As is well known, $\alpha$-constacyclic codes of length $3p^s$ over $\mathcal{R}$ are ideals of the quotient ring $R_{\alpha_1,  \alpha_2,  \alpha_3, \alpha_4} =\frac{\mathcal{R}[x]}{\langle x^{3p^s}-\alpha \rangle}$.
	
    \begin{proposition}\label{2}
        $\alpha=\alpha_1 + \alpha_2 u + \alpha_3 v + \alpha_4 uv$ is non-cube in $\mathcal{R}$ if and only if  $\alpha_1$ is non-cube in $\mathbb{F}_{p^m}$.
    \end{proposition}

    \begin{proof}
        Suppose $\alpha=\alpha_1 + \alpha_2 u + \alpha_3 v + \alpha_4 uv$ is non-cube in $\mathcal{R}$ and assume $\alpha_1= \alpha_1^{\prime^{3}}\in \mathbb{F}_{p^m}$. Then $\alpha_1 + \alpha_2 u + \alpha_3 v + \alpha_4 uv=(\alpha_1^{\prime} +\alpha_2^{\prime}u+ \alpha_3^{\prime} v + \alpha_4^{\prime} uv)^3$, where $\alpha_2^{\prime}=3^{-1} \alpha_2 \alpha_1^{\prime^{-2}},  \alpha_3^{\prime}=3^{-1} \alpha_1^{\prime^{-2}}\alpha_3$ and $\alpha_4^{\prime}=3^{-1} \alpha_1^{\prime^{-2}}(\alpha_4-6\alpha_1^{\prime}\alpha_2^{\prime}\alpha_3^{\prime} )$,  a contradiction. 
        
        Conversely, 
        assume that $\alpha$ is cube in $\mathcal{R}$. Then there exists $\alpha_1^{\prime} +\alpha_2^{\prime}u+ \alpha_3^{\prime} v + \alpha_4^{\prime} uv$, where $\alpha_1^{\prime}, \alpha_2^{\prime}, \alpha_3^{\prime}, \alpha_4^{\prime} \in \mathbb{F}_{p^m}$, such that 
        \begin{align*}
            \alpha=&(\alpha_1^{\prime} +\alpha_2^{\prime}u+ \alpha_3^{\prime} v + \alpha_4^{\prime} uv)^3\\
            =&\alpha_1^{\prime^{3}}+(3\alpha_1^{\prime^{2}}\alpha_2^{\prime})u+(3\alpha_1^{\prime^{2}}\alpha_3^{\prime})v+(3\alpha_1^{\prime^{2}}\alpha_4^{\prime}+6\alpha_1^{\prime}\alpha_2^{\prime}\alpha_3^{\prime})uv.
        \end{align*}
        Thus, $\alpha_1 + \alpha_2 u + \alpha_3 v + \alpha_4 uv=\alpha_1^{\prime^{3}}+(3\alpha_1^{\prime^{2}}\alpha_2^{\prime})u+(3\alpha_1^{\prime^{2}}\alpha_3^{\prime})v+(3\alpha_1^{\prime^{2}}\alpha_4^{\prime}+6\alpha_1^{\prime}\alpha_2^{\prime}\alpha_3^{\prime})uv$. By comparing coefficients, we have $\alpha_1= \alpha_1^{\prime^{3}} \in \mathbb{F}_{p^m}$.
    \end{proof}
\begin{remark}
   We have $\alpha_1={\alpha_0}^{p^{s}}$. By Lemma \ref{2}, $\alpha$ is non-cube in $\mathcal{R}$ if and only if $\alpha_0$ is non-cube in $\mathbb{F}_{p^m}$.
\end{remark}
    \begin{proposition}
        Any non zero  polynomial $c_1x^2+c_2x+c_3 \in \mathbb{F}_{p^m}[x]$ is a unit in $R_{\alpha_1, \alpha_2,  \alpha_3, \alpha_4}$.
    \end{proposition}
    \begin{proof}
       Consider a  non zero polynomial $g(x)=c_1x^2+c_2x+c_3$ in $\mathbb{F}_{p^m}[x]$ i.e., $c_1,c_2,c_3 \in \mathbb{F}_{p^m}$ are not all zeros. 
        \begin{itemize}
            \item If $deg(g(x))=0$. i.e., $c_1=c_2=0, c_3\neq 0$. Thus, $g(x)=c_3$ is a unit.
            \item If $deg(g(x))=1$. i.e., $c_1=0, c_2\neq 0$ and $g(x)=c_2x+c_3$
            In $R_{\alpha_1, \alpha_2,  \alpha_3, \alpha_4}$ we have 
            \begin{align*}
                (x-c_3)^{p^{s}}(x^2+c_3x+c_3^2)^{p^{s}}&=(x^3-c_3^2)^{p^{s}}\\
                &=x^{3p^{s}}-c_3^{2p^{s}}\\
                 &=\alpha_1 +\alpha_2 u+ \alpha_3 v + \alpha_4 uv -c_3^{3p^{s}}\\
                 &=(\alpha_1-c_3^{3p^{s}})+(\alpha_2 u+ \alpha_3v+\alpha_4 uv).
            \end{align*} 
        Not that $\alpha_1-c_2^{3p^{s}}$ is unit in $\mathbb{F}_{p^m}$ as $\alpha_1$ is non-cube in $\mathbb{F}_{p^m}$. Hence, $(\alpha_1-c_2^{3p^{s}})+(\alpha_2 u+ \alpha_3v+\alpha_4 uv)$ is a unit in $\mathcal{R}$. Therefore, 
        \begin{center}
            $(x-c_3)^{-1}=(x-c_3)^{p^{s-1}}(x^2+c_3x+c_3^2)^{p^{s}}(\alpha_1-c_3^{3p^{s}}+\alpha_2 u+ \alpha_3v+\alpha_4 uv)^{-1}$.
        \end{center}
        For any $c_2\neq 0$ in $\mathbb{F}_{p^m}$,  we have 
        \begin{align*}
            (c_2x+c_3)^{-1}&=c_2^{-1}(x-c_2^{-1}(-c_3))^{-1}\\
            &=c_2^{-1}(x-c_2^{-1}(-c_3))^{p^{s-1}}(x^2+c_2^{-1}(-c_3)x+(c_2^{-1}(-c_3))^2)^{p^{s}}(\alpha_1-c_1^{-3p^{s}}c_2^{3p^{s}}+\alpha_2 u+ \alpha_3v+\alpha_4 uv)^{-1}.
        \end{align*}

        \item If $deg(g(x))=2$. i.e., $c_1\neq 0$ and $g(x)=c_1x^2+c_2x+c_3$. In $R_{\alpha_1, \alpha_2,  \alpha_3, \alpha_4}$ we have 
    \begin{align*}
       g^{-1}(x)=&(c_1x^2+c_2x+c_3)^{-1}\\
        =&c^{-1}_1(x^2+c^{-1}_1c_2x+c^{-1}_1c_3)^{-1}\\
        =&c^{-1}_1(x^2+c^{-1}_1c_2x+c^{-1}_1c_3)^{p^s-1}(x^2+c^{-1}_1c_2x+c^{-1}_1c_3)^{-p^s}(x-c^{-1}_1c_2)^{-p^s}(x-c^{-1}_1c_2)^{p^s}\\
        =&c^{-1}_1(x^2+c^{-1}_1c_2x+c^{-1}_1c_3)^{p^s-1}(x-c^{-1}_1c_2)^{p^s}\big[(x^2+c^{-1}_1c_2x+c^{-1}_1c_3)(x-c^{-1}_1c_2)\big]^{-p^s}\\
        =&c^{-1}_1(x^2+c^{-1}_1c_2x+c^{-1}_1c_3)^{p^s-1}(x-c^{-1}_1c_2)^{p^s}\big[x^3+(c^{-1}_1c_3-c^{-2}_1c^{2}_2)x-c^{-2}_1c_2c_3\big]^{-p^s}\\
        =&c^{-1}_1(x^2+c^{-1}_1c_2x+c^{-1}_1c_3)^{p^s-1}(x-c^{-1}_1c_2)^{p^s}\big[x^{3p^s}+(c^{-1}_1c_3-c^{-2}_1c^{2}_2)^{p^s}x^{p^s}-(c^{-2}_1c_2c_3)^{p^s}\big]^{-1}\\
        =&c^{-1}_1(x^2+c^{-1}_1c_2x+c^{-1}_1c_3)^{p^s-1}(x-c^{-1}_1c_2)^{p^s}\\
        &\times \big[ \alpha_1 + \alpha_2 u + \alpha_3 v + \alpha_4 uv+(c^{-1}_1c_3-c^{-2}_1c^{2}_2)^{p^s}x^{p^s}-(c^{-2}_1c_2c_3)^{p^s}\big]^{-1}\\
        =&c^{-1}_1(x^2+c^{-1}_1c_2x+c^{-1}_1c_3)^{p^s-1}(x-c^{-1}_1c_2)^{p^s}\\
        &\times \big[ (\alpha_0-c^{-2}_1c_2c_3 +(c^{-1}_1c_3-c^{-2}_1c^{2}_2)x)^{p^s} +\alpha_2 u + \alpha_3 v + \alpha_4 uv\big]^{-1}
    \end{align*}
So, $g(x)$ is invertible if and only if $\alpha_0-c^{-2}_1c_2c_3 +(c^{-1}_1c_3-c^{-2}_1c^{2}_2)x$ is invertible. Suppose $\alpha_0-c^{-2}_1c_2c_3 +(c^{-1}_1c_3-c^{-2}_1c^{2}_2)x=0$. Then $\alpha_0-c^{-2}_1c_2c_3=0$ and  $c^{-1}_1c_3-c^{-2}_1c^{2}_2=0.$ This implies that $\alpha_0=(c_1^{-1}c_2)^{3}$, a contradiction as $\alpha_0$ is non-cube.
 \end{itemize}
 Therefore, any non zero polynomial $c_1x^2+c_2x+c_3 \in \mathbb{F}_{p^m}[x]$ is a unit in $R_{\alpha_1, \alpha_2,  \alpha_3, \alpha_4}$.
       
    \end{proof}

    \begin{lemma}\label{2.3}
        In $\mathcal{R}$,  $\langle v \rangle \subseteq \langle (x^3-\alpha_0)^{p^{s}}, u \rangle$. Also $x^3-\alpha_0$ is a nilpotent element with nilpotency index $2p^s$ if the characteristic of the ring $R_{\alpha_1, \alpha_2,  \alpha_3, \alpha_4}$ is even and nilpotency index of $x^3-\alpha_0$ is $3p^s$ if the characteristic of the ring $R_{\alpha_1, \alpha_2,  \alpha_3, \alpha_4}$ is odd.
    \end{lemma}
    \begin{proof}
        Let $x^3-\alpha_0 \in \mathcal{R}$. As $p$ is the characteristic and $p\vert\binom{p^s}{i}$,  for every $i$, $1\leq i \leq p^s-1$,  we have 
        \begin{align*}
            (x^3-\alpha_0)^{p^{s}}&=x^{3p^{s}}-\alpha_0^{p^{s}}+\sum_{i=1}^{p^s-1}(-1)^{p^{s}-i}\binom{p^s}{i}(x^{3})^{i}(\alpha_0)^{p^{s}-i}\\
            &=x^{3p^{s}}-\alpha_0^{p^{s}}\\
            &=x^{3p^{s}}-\alpha_1\\
            &=\alpha_2 u+\alpha_3 v+\alpha_4 uv\\
			(x^3-\alpha_0)^{p^{s}}-\alpha_2 u&=v(\alpha_3+\alpha_4 u).
        \end{align*}
		Since $\alpha_3 \neq 0$,  $(\alpha_3+\alpha_4 u)$ is a unit element of $R_{\alpha_1, \alpha_2, \alpha_3,  \alpha_4}$.
		Thus, $\langle v \rangle \subseteq \langle (x^3-\alpha_0)^{p^{s}}, u \rangle$. Now we have $(x^3-\alpha_0)^{2p^{s}}=2\alpha_2 \alpha_3 uv$. If the characteristic is even nilpotency index of $(x^3-\alpha_0)$ is $2p^s$. If the characteristic is odd, we have $(x^3-\alpha_0)^{3p^{s}}=2\alpha_2 \alpha_3 uv(\alpha_2 u+\alpha_3 v+\alpha_4 uv)=0$. Hence,  the nilpotency index of $(x^3-\alpha_0)$ is $3p^s$.
	\end{proof}
	\begin{lemma}\label{4.4}
		Let $x^3-\alpha_0$ be an element of $R_{\alpha_1, \alpha_2, \alpha_3,  \alpha_4}$. Then $(x^3-\alpha_0)^{2p^s}=2\alpha_2 u (x^3-\alpha_0)^{p^s}$,  equivalently,   $\langle (x^3-\alpha_0)^{2p^s}\rangle =\langle u(x^3-\alpha_0)^{p^s} \rangle $.
	\end{lemma}
	\begin{proof}
		From Lemma \ref{2.3}, we have $v=(\alpha_3+\alpha_4 u)^{-1}[(x^3-\alpha_0)^{p^{s}}-\alpha_2 u]$ and $(x^3-\alpha_0)^{2p^{s}}=2\alpha_2 \alpha_3 uv$. This gives
		\begin{align*}
			(x^3-\alpha_0)^{2p^{s}}&=2\alpha_2 \alpha_3 uv\\
			&=2\alpha_2 \alpha_3 u(\alpha_3+\alpha_4 u)^{-1}[(x^3-\alpha_0)^{p^{s}}-\alpha_2 u]\\
			&=2\alpha_2 \alpha_3 u(\alpha_3+\alpha_4 u)^{-1}(x^3-\alpha_0)^{p^{s}}\\
			&=2\alpha_2 u (x^3-\alpha_0)^{p^{s}}.
		\end{align*}
	\end{proof}
	\begin{lemma}
		The ring $R_{\alpha_1, \alpha_2,  \alpha_3, \alpha_4}$ is a local ring with the unique maximal ideal  $\langle (x^3-\alpha_0),  u \rangle $, but $R_{\alpha_1, \alpha_2,  \alpha_3, \alpha_4}$ is not a chain ring.
	\end{lemma}
	\begin{proof}
		Let $h(x)$ be a polynomial in $R_{\alpha_1, \alpha_2, \alpha_3, \alpha_4}$  then it can be expressed as  $h(x)=h_1(x)+uh_2(x)+vh_3(x)+uvh_4(x)$,  where $h_i(x)$, $1\leq i\leq 4$,
		are polynomials of degree less than or equal to $3p^s-1$ over $\mathbb{F}_{p^m}$. Thus, $h(x)$ can be uniquely expressed as
		\begin{align*}
			h(x)=&\sum_{\ell=0}^{p^s-1}(a_{0\ell}x^{2}+b_{0\ell}x+c_{0\ell})(x^3-\alpha_0)^\ell+u\sum_{\ell=0}^{p^s-1}(a_{1\ell}x^{2}+b_{1\ell}x+c_{1\ell})(x^3-\alpha_0)^\ell\\
			&+v\sum_{\ell=0}^{p^s-1}(a_{2\ell}x^{2}+b_{2\ell}x+c_{2\ell})(x^3-\alpha_0)^\ell+uv\sum_{\ell=0}^{p^s-1}(a_{3\ell}x^{2}+b_{3\ell}x+c_{3\ell})(x^3-\alpha_0)^\ell\\
			=&(a_{00}x^2+b_{00}x+c_{00})+\sum_{\ell=1}^{p^s-1}(a_{0\ell}x^{2}+b_{0\ell}x+c_{0\ell})(x^3-\alpha_0)^\ell+u\sum_{\ell=0}^{p^s-1}(a_{1\ell}x^2+b_{1\ell}x+c_{1\ell})(x^3-\alpha_0)^\ell\\
			&+v\begin{pmatrix}
				\sum\limits_{\ell=0}^{p^s-1}(a_{2\ell}x^{2}+b_{2\ell}x+c_{2\ell})(x^3-\alpha_0)^\ell+u\sum\limits_{\ell=0}^{p^s-1}(a_{3\ell}x^{2}+b_{3\ell}x+c_{3\ell})(x^3-\alpha_0)^\ell
			\end{pmatrix},
		\end{align*} 
		where $a_{j\ell},  b_{j\ell}, c_{j\ell} \in \mathbb{F}_{p^m}$,  $j=0, 1, 2, 3$. By Lemma \ref{2.3},  we have $v=(\alpha_3+\alpha_4 u)^{-1}((x^3-\alpha_0)^{p^{s}}-\alpha_2 u )$. Thus, $h(x)$ can be written as
		\begin{align}\label{eqn 4.1}
			\notag h(x)=&(a_{00}x^2+b_{00}x+c_{00})+(x^3-\alpha_0)\sum_{\ell=1}^{p^s-1}(a_{0\ell}x^{2}+b_{0\ell}x+c_{0\ell})(x^3-\alpha_0)^{\ell-1}+u\sum_{\ell=0}^{p^s-1}(a_{1\ell}x^{2}+b_{1\ell}x+c_{1\ell})(x^3-\alpha_0)^\ell\\
			\notag	&+(\alpha_3+\alpha_4 u)^{-1}((x^3-\alpha_0)^{p^{s}}-\alpha_2 u) \Bigg(\sum\limits_{\ell=0}^{p^s-1}(a_{2\ell}x^{2}+b_{2\ell}x+c_{2\ell})(x^3-\alpha_0)^\ell+u\sum\limits_{\ell=0}^{p^s-1}(a_{3\ell}x^{2}+b_{3\ell}x+c_{3\ell})(x^3-\alpha_0)^\ell
			\Bigg)\\
			=&(a_{00}x^2+b_{00}x+c_{00})+(x^3-\alpha_0)\sum\limits_{\ell=1}^{2p^s-1}(a^{\prime}_{0\ell}x^2+b^{\prime}_{0\ell}x+c^{\prime}_{0\ell})(x^3-\alpha_0)^{\ell-1}+u\sum\limits_{\ell=0}^{2p^s-1}(a^{\prime}_{1\ell}x^2+b^{\prime}_{1\ell}x+c^{\prime}_{1\ell})(x^3-\alpha_0)^\ell.
		\end{align} 
		As both $u$ and $(x^3-\alpha_0)$ are nilpotent  elements in $R_{\alpha_1, \alpha_2,  \alpha_3, \alpha_4}$,  the polynomial $h(x)$ is non-unit if and only if $a_{00}=b_{00}=c_{00}=0$. Hence $\langle (x^3-\alpha_0),  u \rangle $ represents the set of all non-units of $R_{\alpha_1, \alpha_2,  \alpha_3, \alpha_4}$,  which shows that $R_{\alpha_1, \alpha_2,  \alpha_3, \alpha_4}$ is a local ring, and the unique maximal ideal is given by $\langle (x^3-\alpha_0),  u \rangle $. It is easy to verify that $x^3-\alpha_0 \notin \langle u \rangle $ and $u\notin \langle (x^3-\alpha_0) \rangle $. Therefore,  the maximal ideal $\langle (x^3-\alpha_0),  u \rangle $ is not principal. By Proposition \ref{prop2.1},  $R_{\alpha_1, \alpha_2,  \alpha_3, \alpha_4}$ is not a chain ring.
		%	 Next we show that $R_{\alpha_1, \alpha_2,  \alpha_3, \alpha_4}$ is not a chain ring. On contrary,  let $u \in  \langle (x^3-\alpha_0) \rangle $ then there exists $l(x) \in R_{\alpha_1, \alpha_2,  \alpha_3, \alpha_4}$ such that 
		%	\begin{align*}
			%		&u+\langle x^{3p^s}-(\alpha_1 +\alpha_2 u+  \alpha_3 v + \alpha_4 uv) \rangle =(x^3-\alpha_0) l(x) +\langle x^{3p^s}-(\alpha_1 +\alpha_2 u+  \alpha_3 v + \alpha_4 uv) \rangle \\
			%		\implies & u-(x^3-\alpha_0) l(x) \in \langle x^{3p^s}-(\alpha_1 +\alpha_2 u+  \alpha_3 v + \alpha_4 uv) \rangle \\
			%		\implies & u-(x^3-\alpha_0) l(x) =  (x^{3p^s}-(\alpha_1 +\alpha_2 u+  \alpha_3 v + \alpha_4 uv) ) m(x) 
			%	\end{align*}
		%	for some $m(x)\in R_{\alpha_1, \alpha_2,  \alpha_3, \alpha_4}[x]$. \colorbox{green}{Put $x^2=\alpha_0$. We get }
		%	\begin{align*}
			%		u&= 
			%	\end{align*}
	\end{proof}
	
\begin{remark}
    The structure of $(\alpha_1 + \alpha_2 u + \alpha_3 v + \alpha_4 uv)$- constacyclic code is similar to the previous section when the characteristic of $R_{\alpha_1, \alpha_2,  \alpha_3, \alpha_4}$ is even. Now look into the case when the characteristic of $R_{\alpha_1, \alpha_2,  \alpha_3, \alpha_4}$ is odd.
\end{remark}

	\begin{theorem}\label{2.7}
		Let $\mathcal{C}=\langle (x^3-\alpha_0)^\ell+u(x^3-\alpha_0)^{t}z(x)\rangle$ and $\Im$ be the smallest integer such that $u (x^3-\alpha_0)^{\Im} \in \mathcal{C}$. Then
		\begin{center}
			$\Im=$
			$\begin{cases}
				min\{\ell, p^s\}  & \text{if}\quad z(x)=0;\\
				min\{\ell, p^s,  2p^s +t-\ell\}  &\text{if}\quad z(x)\neq 0 \quad \text{and} \quad \ell \neq p^s+t;\\
				p^s &\text{if} \quad z(x)\neq 0 \quad \text{and}\quad \ell= p^s+t.
			\end{cases}$
		\end{center}
	\end{theorem}
	\begin{proof}
		We have $\Im \leq \ell $ because 
		\begin{equation*}
			u(x^3-\alpha_0)^\ell=u[(x^3-\alpha_0)^\ell+u(x^3-\alpha_0)^t z(x)] \in \mathcal{C}.
		\end{equation*}
		Suppose $z(x)=0$,  then $\mathcal{C}=\langle (x^3-\alpha_0)^\ell \rangle $,  implies $\Im =\ell$. Let us now examine the case $z(x) \neq 0$,  i.e.,  $z(x)$ is an invertible. As $u(x^3-\alpha_0)^ {\Im} \in \langle (x^3-\alpha_0)^\ell+u(x^3-\alpha_0)^t z(x) \rangle $, there exists a polynomial $f(x)=\sum\limits_{\kappa=0}^{2p^s-1}(a_{0\kappa}x^2+b_{0\kappa}x+c_{0\kappa})(x^3-\alpha_0)^{\kappa}+u\sum\limits_{\kappa=0}^{2p^s-1}(a_{1\kappa}x^2+b_{1\kappa}x+c_{1\kappa})(x^3-\alpha_0)^{\kappa} \in R_{\alpha_1 , \alpha_2,  \alpha_3,  \alpha_4} $, where $a_{0\kappa},  b_{0\kappa},  c_{0\kappa}  a_{1\kappa},  b_{1\kappa}, c_{1\kappa}  \in \mathbb{F}_{p^m}$ such that $u(x^3-\alpha_0)^{\Im} =f(x)[(x^3-\alpha_0)^\ell+u(x^3-\alpha_0)^t z(x)]$. Then 
		\begin{align*}
			u(x^3-\alpha_0)^{\Im} =&
			\begin{bmatrix}
				\sum\limits_{\kappa=0}^{2p^s-1}(a_{0\kappa}x^2+b_{0\kappa}x+c_{0\kappa})(x^3-\alpha_0)^{\kappa}+u\sum\limits_{\kappa=0}^{2p^s-1}(a_{1\kappa}x^2+b_{1\kappa}x+c_{1\kappa})(x^3-\alpha_0)^{\kappa}
			\end{bmatrix}\\
			& \times [(x^3-\alpha_0)^\ell +u(x^3-\alpha_0)^t z(x)] \\
			=&(x^3-\alpha_0)^\ell \sum\limits_{\kappa=0}^{2p^s-1}(a_{0\kappa}x^2+b_{0\kappa}x+c_{0\kappa})(x^3-\alpha_0)^{\kappa}\\
			&+u(x^3-\alpha_0)^\ell \sum\limits_{\kappa=0}^{2p^s-1}(a_{1\kappa}x^2+b_{1\kappa}x+c_{1\kappa})(x^3-\alpha_0)^{\kappa}\\
			&+u(x^3-\alpha_0)^t z(x)\sum\limits_{\kappa=0}^{2p^s-1}(a_{0\kappa}x^2+b_{0\kappa}x+c_{0\kappa})(x^3-\alpha_0)^{\kappa}\\
			=&2\alpha_2 u(x^3-\alpha_0)^{p^s}
			\sum\limits_{\kappa=0}^{\ell-1}(a_{0\kappa}x^2+b_{0\kappa}x+c_{0\kappa})(x^3-\alpha_0)^{\kappa}\\
			&+u(x^3-\alpha_0)^\ell \sum\limits_{\kappa=0}^{2p^s-\ell-1}(a_{1\kappa}x^2+b_{1\kappa}x+c_{1\kappa})(x^3-\alpha_0)^{\kappa}\\
			&+u(x^3-\alpha_0)^{2p^s+t-\ell} z(x)\sum\limits_{\kappa=0}^{\ell-1}(a_{0\kappa}x^2+b_{0\kappa}x+c_{0\kappa})(x^3-\alpha_0)^{\kappa}.
		\end{align*}
		Thus,  $\Im \geq \min\{\ell, p^s,  2p^s+t-\ell\}$. If $z(x)=0$,  then $\Im \geq \min\{\ell, p^s\}$. Moreover by Lemma \ref{4.4}
		\begin{equation*}
			[(x^3-\alpha_0)^\ell +u(x^3-\alpha_0)^t z(x)](x^3-\alpha_0)^{2p^s-\ell}=2\alpha_2 u(x^3-\alpha_0)^{p^s}+u(x^3-\alpha_0)^{2p^s+t-\ell}z(x) \in \mathcal{C}.
		\end{equation*}
		If $z(x)=0$,  then $\Im \leq p^s$.
		Suppose $z(x)\neq 0$ and $\ell\neq p^s+t$, then $\Im \leq \min\{ \ell, p^s, 2p^s+t-\ell\}$.
		Suppose $z(x) \neq 0$ and $\ell=p^s+t$, then $\Im= p^s$.
	\end{proof}
	The next step is determining the number of codewords in a constacyclic code.
	Let $\psi_u$ denote reduction modulo $u$ map from $R_{\alpha_1,  \alpha_2,  \alpha_3, \alpha_4}$ to $\frac{(\mathbb{F}_{p^m} + v\mathbb{F}_{p^m})[x]}{\langle x^{3p^s}-(\alpha_1 + \alpha_3 v) \rangle}$. The map $\psi_u$ is a surjective ring homomorphism. Let
	$\mathcal{C}$ be an ideal of $R_{\alpha_1,  \alpha_2,  \alpha_3, \alpha_4}$. We define the torsion of $\mathcal{C}$ by $Tor(\mathcal{C})=\{a(x) \in \frac{(\mathbb{F}_{p^m} + v\mathbb{F}_{p^m})[x]}{\langle x^{3p^s}-(\alpha_1 + \alpha_3 v) \rangle} : ua(x)\in \mathcal{C}\}$ and residue of $\mathcal{C}$ by  $Res(\mathcal{C})=\psi_u(\mathcal{C}).$ Clearly, $Tor(\mathcal{C})$ and $Res(\mathcal{C})$  are ideals of $\frac{(\mathbb{F}_{p^m} + v\mathbb{F}_{p^m})[x]}{\langle x^{3p^s}-(\alpha_1 + \alpha_3 v) \rangle}$ and $Tor(\mathcal{C}) \cong Ker(\psi_u\vert \mathcal{C} )$. Therefore, $\vert \mathcal{C}\vert =\vert Res(\mathcal{C})\vert \vert Tor(\mathcal{C})\vert  $. We calculate the number of codewords in $Tor(\mathcal{C})$ and $Res(\mathcal{C})$ to get the number of codewords in $\mathcal{C}$.  We know from \cite{chen2016constacyclic} that any ideal of $\frac{(\mathbb{F}_{p^m} + v\mathbb{F}_{p^m})[x]}{\langle x^{3p^s}-(\alpha_1 + \alpha_3 v) \rangle}$ is of the form  $\langle (x^3-\alpha_0)^\ell \rangle $, where
	$0 \leq \ell \leq 2p^s$ and it has $p^{3m(2p^s-\ell)}$ codewords.
	%	\begin{proposition}\colorbox{green}{[ ]}
		%		Let $\mathcal{C}$ be a linear code of length $n$ over $R$ whose torsion and residue codes are $Tor(\mathcal{C})$ and $Res(\mathcal{C})$. Then $\vert \mathcal{C}\vert=\vert Tor(\mathcal{C})\vert \cdot \vert Res(\mathcal{C}) \vert$.
		%	\end{proposition}

	\begin{theorem}
		Let $\mathcal{C}$ be a $(\alpha_1 +\alpha_2 u+  \alpha_3 v + \alpha_4 uv)$-constacyclic code of length $3p^s$ over $\mathcal{R}$ and let $\eta_\mathcal{C}$ denote the number of codewords in $\mathcal{C}$. Then
		\begin{enumerate}
			\item Type A: Trivial ideals
			\begin{itemize}
				\item If $\mathcal{C}=\langle 0 \rangle$,  then $Res(\mathcal{C})=\langle 0 \rangle=Tor(\mathcal{C})$ and $\eta_\mathcal{C}=1$.
				\item If $\mathcal{C}=\langle 1 \rangle$,  then $Res(\mathcal{C})=\langle 1 \rangle=Tor(\mathcal{C})$ and $\eta_\mathcal{C}=p^{12mp^{s}}$.
			\end{itemize}
			\item Type B: Principal ideals with nonmonic generators
			\begin{itemize}
				\item If $\mathcal{C}=\langle u(x^3-\alpha_0)^\ell \rangle $,  where $0\leq  \ell \leq 2p^s-1$,  then  $Tor(\mathcal{C})=\langle (x^3-\alpha_0)^\ell \rangle $, $Res(\mathcal{C})=\langle 0 \rangle$ and $\eta_\mathcal{C}=p^{3m(2p^{s}-\ell)}$.
			\end{itemize}
			\item Type C: Principal ideals with monic generators
			\begin{itemize}
				\item If $\mathcal{C}=\langle (x^3-\alpha_0)^\ell \rangle $,  where $1\leq  \ell \leq 2p^s-1$,  then $Res(\mathcal{C})=\langle (x^3-\alpha_0)^\ell \rangle $ and $Tor(\mathcal{C})=\langle (x^3-\alpha_0)^{\Im} \rangle $ and 
				\begin{center}
					$\eta_\mathcal{C}=$
					$\begin{cases}
						p^{6m(2p^s-\ell)} & \text{if}\quad 1\leq \ell\leq p^s;\\
						p^{3m(3p^s-\ell)} & \text{if}\quad p^s < \ell \leq 2p^s-1.
					\end{cases}$
				\end{center}
				\item Let $\mathcal{C}=\langle (x^3-\alpha_0)^\ell +u(x^3-\alpha_0)^t z(x) \rangle $,  where $1\leq  \ell \leq 2p^s-1$,  $0\leq  t < \ell  $ and $z(x)$ is an invertible. If $\ell \neq p^s+t$ then  $Res(\mathcal{C})=\langle (x^3-\alpha_0)^\ell \rangle $ and $Tor(\mathcal{C})=\langle (x^3-\alpha_0)^{\Im} \rangle $, where $\Im$ is the smallest integer such that $u(x^3-\alpha_0)^{\Im} \in \mathcal{C}$,  which is given by $\Im = min\{ p^s, \ell, 2p^s+t-\ell\}$. Thus,
				\begin{center}
					$\eta_\mathcal{C}=$
					$\begin{cases}
						p^{6m(2p^s-\ell)} & \text{if}\quad 1\leq \ell\leq p^s;\\
						p^{3m(3p^s-\ell)} & \text{if}\quad p^s \leq \ell < p^s+t;\\
						p^{3m(2p^s-t)} & \text{if} \quad  p^s+ t < \ell \leq 2p^s-1.
					\end{cases}$
				\end{center} 
				\item Let $\mathcal{C}=\langle (x^3-\alpha_0)^{p^s+t} +u(x^3-\alpha_0)^t z(x) \rangle $,  where $0\leq  t < p^s  $ and $z(x)$ is an invertible. Then  $Res(\mathcal{C})=\langle (x^3-\alpha_0)^\ell \rangle $ and  $Tor(\mathcal{C})=\langle (x^3-\alpha_0)^{p^s} \rangle $. Thus, $\eta_\mathcal{C}=p^{3m(2p^s-t)}$.
			\end{itemize}
			\item Type D: Non-principal ideals \\
			If $\mathcal{C}=\langle (x^3-\alpha_0)^\ell +u(x^3-\alpha_0)^t z(x),  u (x^3-\alpha_0)^{\mu} \rangle $,  where $1\leq  \ell \leq 2p^s-1$,  $0\leq  t < \ell  $ and $z(x)$ is 0 or $z(x)$ is an invertible, and deg $(z(x))\leq\mu-t-1$ and
			\begin{center}
				$\mu < \Im=$
				$\begin{cases}
					min\{\ell, p^s\}  & \text{if}\quad z(x)=0;\\
					min\{\ell,p^s,  2p^s-\ell+t\}  &\text{if}\quad z(x)\neq 0\quad \text{and}\quad \ell \neq p^s+t;\\
					p^s &\text{if}\quad z(x)\neq 0 \quad\text{and}\quad \ell= p^s+t.
				\end{cases}$
			\end{center}
			Then $Res(\mathcal{C})=\langle (x^3-\alpha_0)^\ell \rangle $ and $Tor(\mathcal{C})=\langle (x^3-\alpha_0)^{\mu} \rangle $. Thus, $\eta_\mathcal{C}=p^{3m(4p^s-\ell-\mu)}$.
		\end{enumerate}
	\end{theorem}
	\subsection{Duals of constacyclic codes}
	\begin{lemma}
		Let $\pi(x)=(x^3-\alpha_0)^{j}+u\sum\limits_{i=0}^{t}(a_{i}x^2+b_{i}x+c_{i})(x^3-\alpha_0)^{i}$ be a polynomial over $R_{\alpha_1, \alpha_2,  \alpha_3, \alpha_4}$,  where $t<j$. Then
		\begin{equation*}
			\pi^*(x)=(-\alpha_0)^{j}(x^3-\alpha_0^{-1})^{j}+u\sum\limits_{i=0}^{t}(c_{i}x^2+b_{i}x+a_{i})(-\alpha_0)^{i}(x^3-\alpha_0^{-1})^ix^{3j-3i-2}.
		\end{equation*}
	\end{lemma}
	\begin{proof}
		By Lemma \ref{2.6},
		\begin{equation*}
			[(x^3-\alpha_0)^{k}]^{*}=[(x^3-\alpha_0)^{*}]^{k}=(-\alpha_0 x^3+1)^k=(-\alpha_0)^k(x^3-\alpha_0^{-1})^k.
		\end{equation*}
		Applying Lemma \ref{2.6}. again, we have 
		\begin{align*}
			\pi^*(x)&=[(x^3-\alpha_0)^{j}]^{*}+u\sum\limits_{i=0}^{t}(a_{i}x^2+b_{i}x+c_{i})^*[(x^3-\alpha_0)^{i}]^{*}x^{3j-3i-2}\\
			&=(-\alpha_0)^{j}(x^3-\alpha_0^{-1})^{j}+u\sum\limits_{i=0}^{t}(c_{i}x^2+b_{i}x+a_{i})(-\alpha_0)^{i}(x^3-\alpha_0^{-1})^ix^{3j-3i-2}.
		\end{align*}
	\end{proof}
	\begin{lemma}{\label{5.2}}
		If $\mathcal{C}=\langle(x^3-\alpha_0)^\ell+u(x^3-\alpha_0)^tz(x),  u(x^3-\alpha_0)^\mu \rangle$,  then the smallest positive integer $\varepsilon$ such that $u(x^3-\alpha_0)^{\varepsilon} \in \mathcal{A}(\mathcal{C})$ is $2p^s-\ell$. 
	\end{lemma}
	\begin{proof}
		Suppose
		\begin{equation*}
			[(x^3-\alpha_0)^\ell+u(x^3-\alpha_0)^tz(x)]u(x^3-\alpha_0)^{\varepsilon}=0.
		\end{equation*}
		Then from Lemma \ref{2.3} and \ref{4.4},  we have $\ell+\varepsilon \geq 2p^s$,  i.e.,  $\varepsilon \geq 2p^s-\ell$.
	\end{proof}
	\begin{theorem}
		Let $\mathcal{C}=\langle u(x^3-\alpha_0)^\ell \rangle $ be a  $(\alpha_1 +\alpha_2 u+  \alpha_3 v + \alpha_4 uv)$-constacyclic code of length $3p^s$ over $\mathcal{R}$. Then $\mathcal{C}^{\perp}=\langle (x^3-\alpha_0^{-1})^{2p^s-\ell},  u \rangle$.
	\end{theorem}
	\begin{proof}
		As $\mathcal{C}\subseteq \langle u \rangle $ and $\mathcal{C} \subseteq \langle (x^3-\alpha_0)^\ell \rangle $,  we have that $\langle (x^3-\alpha_0^{-1})^{2p^s-\ell }\rangle =\langle (x^3-\alpha_0)^\ell \rangle ^{\perp} \subseteq \mathcal{C}^{\perp}$ and $\langle u \rangle =\langle u \rangle ^{\perp} \subseteq \mathcal{C}^{\perp}$. So $\langle (x^3-\alpha_0^{-1})^{2p^s-\ell },  u \rangle \subseteq \mathcal{C}^{\perp}$. We have $\vert \mathcal{C}\vert=p^{3m(2p^s-\ell)}$ and \begin{equation*}
			\vert \langle (x^3-\alpha_0^{-1})^{2p^s-\ell },  u \rangle \vert =p^{3m(2p^s+\ell)}.
		\end{equation*}
		Therefore, 
		\begin{equation*}
			\vert \mathcal{C}^{\perp} \vert=\frac{\vert \mathcal{R} \vert ^{3p^s}}{\vert \mathcal{C} \vert}=\frac{p^{12mp^s}}{p^{3m(2p^s-\ell)}}=p^{3m(2p^s+\ell)}=\vert \langle (x^3-\alpha_0^{-1})^{2p^s-\ell },  u \rangle \vert.
		\end{equation*}
		Hence,  $\mathcal{C}^{\perp}=\langle (x^3-\alpha_0^{-1})^{2p^s-\ell },  u \rangle$.
	\end{proof}
	\begin{theorem}
		Let $\mathcal{C}$ be a $(\alpha_1 +\alpha_2 u+  \alpha_3 v + \alpha_4 uv)$-constacyclic code associate to the ideal $\mathcal{C}=\langle (x^3-\alpha_0)^\ell +u(x^3-\alpha_0)^t z(x) \rangle $,  where $z(x)$ is 0 or an invertible. Then the dual code $\mathcal{C}^\perp $ associate to the ideal $\mathcal{A}(\mathcal{C})^*$ is determined as follows:
		\begin{enumerate}
			\item If $z(x)$ is 0 and $1\leq \ell\leq p^s$,  then $\mathcal{A}(\mathcal{C})^* = \langle \phi_1^*(x)  \rangle$,  where $\phi_1^*(x)=(-\alpha_0)^{2p^s-\ell}(x^3-\alpha_0^{-1})^{2p^s-\ell }-2\alpha_2 u (-\alpha_0)^{p^s-\ell}(x^3-\alpha_0^{-1})^{p^s-\ell } x^{3p^s}  $.
			\item If $z(x)$ is 0 and $p^s < \ell \leq 2p^s-1$,  then $\mathcal{A}(\mathcal{C})^* =\langle \phi_2^*(x), u(x^3-\alpha_0^{-1})^{2p^s-\ell }  \rangle$,  where $ \phi_2^*(x)= (-\alpha_0)^{p^s}(x^3-\alpha_0^{-1})^{p^s}-2\alpha_2 u x^{3p^s} $.
			\item If $z(x)$ is an invertible and $1\leq \ell\leq p^s$,  $\ell \neq p^s+t$,   then $\mathcal{A}(\mathcal{C})^*=\langle \phi^*_3(x) \rangle $,  where $\phi^*_3(x)=(-\alpha_0)^{2p^s-\ell}(x^3-\alpha_0^{-1})^{2p^s-\ell}-2\alpha_2 u (-\alpha_0)^{p^s-\ell}(x^3-\alpha_0^{-1})^{p^s-\ell}x^{3p^s}-u(-\alpha_0)^{2p^s-2\ell+t}(x^3-\alpha_0^{-1})^{2p^s-2\ell+t} \\
   \sum \limits_{\kappa=0}^{\ell-t-1}(z_{2\kappa}x^2+z_{1\kappa}x+z_{0\kappa})(-\alpha_0)^{\kappa} (x^3-\alpha_0^{-1})^{\kappa} x^{3\ell-3t-3\kappa-2}$. 
			\item If $z(x)$ is an invertible and $p^s\leq \ell \leq p^s+t$,  then $\mathcal{A}(\mathcal{C})^*=\langle \phi^*_4(x),  (x^3-\alpha_0^{-1})^{2p^s-\ell} \rangle $,  where $\phi^*_4(x)=(-\alpha_0)^{p^s}(x^3-\alpha_0^{-1})^{p^s}-2\alpha_2 u x^{3p^s}-u(-\alpha_0)^{p^s+t-\ell}(x^3-\alpha_0^{-1})^{p^s+t-\ell} \sum \limits_{\kappa=0}^{p^s-t-1}(z_{2\kappa}x^2+z_{1\kappa}x+z_{0\kappa})(-\alpha_0)^{\kappa}(x^3-\alpha_0^{-1})^{\kappa} x^{3\ell-3t-3\kappa-2} $.
			\item If $z(x)$ is an invertible and $p^s+ t < \ell \leq 2p^s-1$,  then $\mathcal{A}(\mathcal{C})^*=\langle \phi^*_5(x),  (x^3-\alpha_0^{-1})^{2p^s-\ell} \rangle $,  where $\phi^*_5(x)= (-\alpha_0)^{\ell-t}(x^3-\alpha_0^{-1})^{\ell-t}-2\alpha_2 u(-\alpha_0)^{\ell-p^s-t}(x^3-\alpha_0^{-1})^{\ell-p^s-t}x^{3p^s}-u\sum \limits_{\kappa=0}^{2p^s-\ell-1}(z_{2\kappa}x^2+z_{1\kappa}x+z_{0\kappa})(-\alpha_0)^{\kappa}(x^3-\alpha_0^{-1})^{\kappa} x^{3\ell-3t-3\kappa-2}$.
			\item If $z(x)$ is an invertible and $\ell=p^s+t$,   then $\mathcal{A}(\mathcal{C})^*=\langle \phi^*_6(x) ,  (x^3-\alpha_0^{-1})^{2p^s-\ell} \rangle$,  where $\phi^*_6(x)=(-\alpha_0)^{p^s}(x^3-\alpha_0^{-1})^{p^s}-2\alpha_2 u x^{3p^s}- u\sum \limits_{\kappa=0}^{p^s-t-1}(z_{2\kappa}x^2+z_{1\kappa}x+z_{0\kappa})(-\alpha_0)^{\kappa}(x^3-\alpha_0^{-1})^{\kappa} x^{p^s-2\kappa -1} $.
		\end{enumerate}
	\end{theorem}
	\begin{proof}
		By Lemma \ref{5.2}, $2p^s-\ell$ is the smallest positive integer $\varepsilon$ such that $u(x^3-\alpha_0)^{\varepsilon} \in \mathcal{A}(\mathcal{C})$. If $(x^3-\alpha_0)^{\rho}+u(x^3-\alpha_0)^{\xi}r(x)\in \mathcal{A}(\mathcal{C})$ where $r(x)$ is 0 or an invertible then [$(x^3-\alpha_0)^{\rho}+u(x^3-\alpha_0)^{\xi}r(x)][(x^3-\alpha_0)^{\ell}+u(x^3-\alpha_0)^{t}z(x)]=(x^3-\alpha_0)^{\rho+\ell}+u(x^3-\alpha_0)^{\rho+t}z(x)+u(x^3-\alpha_0)^{\xi +\ell}r(x)=0$.
		\begin{enumerate}
			\item  If $z(x)$ is 0 and $1\leq \ell\leq p^s$,  then $\mathcal{C}=\langle (x^3-\alpha_0)^\ell \rangle $. If we choose $\rho=2p^s-\ell$,  $\xi=p^s-\ell$ and $r(x)=-2\alpha_2$ then by Lemma \ref{4.4},  $\phi_1(x)=(x^3-\alpha_0)^{2p^s-\ell}-2\alpha_2 u(x^3-\alpha_0)^{p^s-\ell} \in \mathcal{A}(\mathcal{C})$. Therefore, $\langle \phi_1(x) \rangle \subseteq \mathcal{A}(\mathcal{C})$. Furthermore, 
			\begin{equation*}
				p^{6m\ell}=\vert \langle \phi_1(x) \rangle \vert\leq \vert \mathcal{A}(\mathcal{C}) \vert =\vert \mathcal{A}(\mathcal{C})^{*} \vert =\vert \mathcal{C}^{\perp}\vert=\frac{p^{12mp^s}}{\vert \mathcal{C} \vert}=\frac{p^{12mp^s}}{p^{6m(2p^s-\ell)}}=p^{6m\ell}.
			\end{equation*}
			Thus, $\langle \phi_1(x) \rangle=\mathcal{A}(\mathcal{C})$ and  $ \mathcal{A}(\mathcal{C})^{*} = \langle \phi_1^*(x)  \rangle$,  where $\phi_1^*(x)=(-\alpha_0)^{2p^s-\ell}(x^3-\alpha_0^{-1})^{2p^s-\ell }-2\alpha_2 u (-\alpha_0)^{p^s-\ell}(x^3-\alpha_0^{-1})^{p^s-\ell } x^{3p^s}$.
			\item If $z(x)$ is 0 and $p^s < \ell \leq 2p^s-1$,  then $\mathcal{C}=\langle (x^3-\alpha_0)^\ell \rangle $. If we choose $\rho=p^s$,  $\xi=0$ and $r(x)=-2\alpha_2$ then by Lemma \ref{4.4},  $\phi_2(x)=(x^3-\alpha_0)^{p^s}-2\alpha_2 u \in \mathcal{A}(\mathcal{C})$. Therefore, $\langle \phi_2(x),  u(x^3-\alpha_0)^{2p^s-\ell} \rangle \subseteq \mathcal{A}(\mathcal{C})$. Furthermore, 
		\begin{align*}
				p^{3m(p^s+\ell)}=\vert \langle \phi_2(x), u(x^3-\alpha_0)^{2p^s-\ell}  \rangle \vert&\leq  \vert \mathcal{A}(\mathcal{C}) \vert 
				=\vert \mathcal{A}(\mathcal{C})^{*} \vert 
				=\vert \mathcal{C}^{\perp}\vert
				=\frac{p^{12mp^s}}{\vert \mathcal{C} \vert}
				=\frac{p^{12mp^s}}{p^{3m(3p^s-\ell)}}=p^{3m(p^s+\ell)}.
			\end{align*}
			Thus, $\langle \phi_2(x), u(x^3-\alpha_0)^{2p^s-\ell}  \rangle=\mathcal{A}(\mathcal{C})$ and  $ \mathcal{A}(\mathcal{C})^{*} = \langle \phi_2^*(x), u(x^3-\alpha_0^{-1})^{2p^s-\ell }  \rangle$,  where $ \phi_2^*(x)= (-\alpha_0)^{p^s}(x^3-\alpha_0^{-1})^{p^s}-2\alpha_2 u x^{3p^s} $.
			\item  If $z(x)$ is an invertible and $1\leq \ell\leq p^s$,  $\ell \neq p^s+t$,  then $\mathcal{C}=\langle (x^3-\alpha_0)^\ell+u(x^3-\alpha_0)^t z(x) \rangle $. If we choose $\rho=2p^s-\ell$,  $\xi=p^s-\ell$ and $r(x)=-2\alpha_2- z(x)(x^3-\alpha_0)^{p^s+t-\ell}$ then by Lemma \ref{4.4},  $\phi_3(x)=(x^3-\alpha_0)^{2p^s-\ell}+u(x^3-\alpha_0)^{p^s-\ell}(-2\alpha_2 -z(x)(x^3-\alpha_0)^{p^s+t-\ell}) \in \mathcal{A}(\mathcal{C})$. Therefore, $\langle \phi_3(x) \rangle \subseteq \mathcal{A}(\mathcal{C})$. Furthermore, 
			\begin{equation*}
				p^{6m\ell}=\vert \langle \phi_3(x) \rangle \vert\leq \vert \mathcal{A}(\mathcal{C}) \vert =\vert \mathcal{A}(\mathcal{C})^{*} \vert =\vert \mathcal{C}^{\perp}\vert=\frac{p^{12mp^s}}{\vert \mathcal{C} \vert}=\frac{p^{12mp^s}}{p^{4m(2p^s-\ell)}}=p^{6m\ell}.
			\end{equation*}
			Thus, $\langle \phi_3(x) \rangle=\mathcal{A}(\mathcal{C})$. 
			Let $z(x)=\sum\limits_{\kappa}^{}(z_{0\kappa}x^2+z_{1\kappa}x+z_{2\kappa})(x^3-\alpha_0)^{\kappa}$ where $z_{0\kappa},  z_{1\kappa}, z_{2\kappa} \in \mathbb{F}_{p^m}$ and $z_{00}x^2+z_{10}x+z_{20} \neq 0$. Here $\kappa \leq \Im-t-1=\ell-t-1$.
			\begin{align*}
				\phi_3(x)&=(x^3-\alpha_0)^{2p^s-\ell}+u(x^3-\alpha_0)^{p^s-\ell}(-2\alpha_2 -z(x)(x^3-\alpha_0)^{p^s+t-\ell})\\
				&=(x^3-\alpha_0)^{2p^s-\ell}-2\alpha_2 u(x^3-\alpha_0)^{p^s-\ell} -uz(x)(x^3-\alpha_0)^{2p^s+t-2\ell}\\
				&=(x^3-\alpha_0)^{2p^s-\ell}-2\alpha_2 u(x^3-\alpha_0)^{p^s-\ell} -u(x^3-\alpha_0)^{2p^s+t-2\ell}\sum\limits_{\kappa=0}^{\ell-t-1}(z_{0\kappa}x^2+z_{1\kappa}x+z_{2\kappa})(x^3-\alpha_0)^{\kappa}.
			\end{align*}
			Then 
			$\mathcal{A}(\mathcal{C})^{*}=\langle \phi^*_3(x) \rangle $,  where $\phi^*_3(x)=(-\alpha_0)^{2p^s-\ell}(x^3-\alpha_0^{-1})^{2p^s-\ell}-2\alpha_2 u (-\alpha_0)^{p^s-\ell}(x^3-\alpha_0^{-1})^{p^s-\ell}x^{3p^s}-u(-\alpha_0)^{2p^s+t-2\ell}(x^3-\alpha_0^{-1})^{2p^s+t-2\ell} \sum \limits_{\kappa=0}^{\ell-t-1}(z_{2\kappa}x^2+z_{1\kappa}x+z_{0\kappa})(-\alpha_0)^{\kappa}(x^3-\alpha_0^{-1})^{\kappa} x^{3\ell-3t-3\kappa-2}$. 
			\item If $z(x)$ is an invertible and $p^s\leq \ell < p^s+t$,  then $\mathcal{C}=\langle (x^3-\alpha_0)^\ell+u(x^3-\alpha_0)^t z(x) \rangle $. As above $\phi_4(x)=(x^3-\alpha_0)^{p^s}+u[-2\alpha_2-(x^3-\alpha_0)^{p^s+t-\ell}z(x)] \in \mathcal{A}(\mathcal{C})$. Therefore, $\langle \phi_2(x),  u(x^3-\alpha_0)^{2p^s-\ell} \rangle \subseteq \mathcal{A}(\mathcal{C})$. Furthermore, 
			\begin{equation*}
				p^{3m(p^s+\ell)}=\vert \langle \phi_4(x), u(x^3-\alpha_0)^{2p^s-\ell}  \rangle \vert\leq \vert \mathcal{A}(\mathcal{C}) \vert =\vert \mathcal{A}(\mathcal{C})^{*} \vert =\vert \mathcal{C}^{\perp}\vert=\frac{p^{12mp^s}}{\vert \mathcal{C} \vert}=\frac{p^{12mp^s}}{p^{3m(3p^s-\ell)}}=p^{3m(p^s+\ell)}.
			\end{equation*}
			Thus, $\langle \phi_4(x), u(x^3-\alpha_0)^{2p^s-\ell}  \rangle=\mathcal{A}(\mathcal{C})$.
			Let $z(x)=\sum\limits_{\kappa}^{}(z_{0\kappa}x^2+z_{1\kappa}x+z_{2\kappa})(x^3-\alpha_0)^{\kappa}$ where $z_{0\kappa},  z_{1\kappa}, z_{2\kappa} \in \mathbb{F}_{p^m}$ and $z_{00}x^2+z_{10}x+z_{20} \neq 0$. Here $\kappa \leq\Im-t-1=p^s-t-1$.
			\begin{align*}
				\phi_4(x)&=(x^3-\alpha_0)^{p^s}+u[-2\alpha_2-(x^3-\alpha_0)^{p^s+t-\ell}z(x)]\\
				&=(x^3-\alpha_0)^{p^s}-2\alpha_2 u-u(x^3-\alpha_0)^{p^s+t-\ell}\sum\limits_{\kappa=0}^{p^s-t-1}(z_{0\kappa}x^2+z_{1\kappa}x+z_{2\kappa})(x^3-\alpha_0)^{\kappa}.
			\end{align*}
			Then $ \mathcal{A}(\mathcal{C})^{*}=\langle \phi^*_4(x),  (x^3-\alpha_0^{-1})^{2p^s-\ell} \rangle $,  where $\phi^*_4(x)=(-\alpha_0)^{p^s}(x^3-\alpha_0^{-1})^{p^s}-2\alpha_2 u x^{3p^s}-u(-\alpha_0)^{p^s+t-\ell}(x^3-\alpha_0^{-1})^{p^s+t-\ell} \sum \limits_{\kappa=0}^{p^s-t-1}(z_{2\kappa}x^2+z_{1\kappa}x+z_{0\kappa})(-\alpha_0)^{\kappa}(x^3-\alpha_0^{-1})^{\kappa} x^{3\ell-3t-3\kappa-2} $.
			\item If $z(x)$ is an invertible and $p^s+ t < \ell \leq 2p^s-1$,  then $\mathcal{C}=\langle (x^3-\alpha_0)^\ell+u(x^3-\alpha_0)^t z(x) \rangle $. As above $\phi_5(x)=(x^3-\alpha_0)^{\ell-t}+u(x^3-\alpha_0)^{\ell-p^s-t}[-2\alpha_2-(x^3-\alpha_0)^{p^s+t-\ell}z(x)] \in \mathcal{A}(\mathcal{C})$. Therefore, $\langle \phi_5(x),  u(x^3-\alpha_0)^{2p^s-\ell} \rangle \subseteq \mathcal{A}(\mathcal{C})$. Furthermore, 
			\begin{equation*}
				p^{3m(2p^s+t)}=\vert \langle \phi_5(x), u(x^3-\alpha_0)^{2p^s-\ell}  \rangle \vert\leq \vert \mathcal{A}(\mathcal{C}) \vert =\vert \mathcal{A}(\mathcal{C})^{*} \vert =\vert \mathcal{C}^{\perp}\vert=\frac{p^{12mp^s}}{\vert \mathcal{C} \vert}=\frac{p^{12mp^s}}{p^{3m(2p^s-t)}}=p^{3m(2p^s+t)}.
			\end{equation*}
			Thus, $\langle \phi_5(x), u(x^3-\alpha_0)^{2p^s-\ell}  \rangle=\mathcal{A}(\mathcal{C})$.
			Let $z(x)=\sum\limits_{\kappa}^{}(z_{0\kappa}x^2+z_{1\kappa}x+z_{2\kappa})(x^3-\alpha_0)^{\kappa}$, where $z_{0\kappa},  z_{1\kappa}, z_{2\kappa} \in \mathbb{F}_{p^m}$ and $z_{00}x^2+z_{10}x+z_{20} \neq 0$. Here $\kappa \leq\Im-t-1=2p^s-\ell-1$.
			\begin{align*}
				\phi_5(x)&=(x^3-\alpha_0)^{\ell-t}+u(x^3-\alpha_0)^{\ell-p^s-t}[-2\alpha_2-(x^3-\alpha_0)^{p^s+t-\ell}z(x)]\\
				&=(x^3-\alpha_0)^{\ell-t}-2\alpha_2 u(x^3-\alpha_0)^{\ell-p^s-t}-u\sum\limits_{\kappa=0}^{2p^s-\ell-1}(z_{0\kappa}x^2+z_{1\kappa}x+z_{2\kappa})(x^3-\alpha_0)^{\kappa}.
			\end{align*}
			Then  $ \mathcal{A}(\mathcal{C})^{*}=\langle \phi^*_5(x),  (x^3-\alpha_0^{-1})^{2p^s-\ell} \rangle $,  where $\phi^*_5(x)= (-\alpha_0)^{\ell-t}(x^3-\alpha_0^{-1})^{\ell-t}-2\alpha_2 u(-\alpha_0)^{\ell-p^s-t}(x^3-\alpha_0^{-1})^{\ell-p^s-t}x^{3p^s}-u\sum \limits_{\kappa=0}^{2p^s-\ell-1}(z_{2\kappa}x^2+z_{1\kappa}x+z_{0\kappa})(-\alpha_0)^{\kappa}(x^3-\alpha_0^{-1})^{\kappa} x^{3\ell-3t-3\kappa-2}$.
			\item If $z(x)$ is an invertible and $\ell=p^s+t$,  then $\mathcal{C}=\langle (x^3-\alpha_0)^\ell+u(x^3-\alpha_0)^t z(x) \rangle $. As above $\phi_6(x)=(x^3-\alpha_0)^{p^s}-2\alpha_2 u-uz(x) \in \mathcal{A}(\mathcal{C})$. Therefore, $\langle \phi_6(x),  u(x^3-\alpha_0)^{2p^s-\ell} \rangle \subseteq \mathcal{A}(\mathcal{C})$. Furthermore, 
			\begin{equation*}
				p^{3m(2p^s+t)}=\vert \langle \phi_6(x), u(x^3-\alpha_0)^{2p^s-\ell}  \rangle \vert\leq \vert \mathcal{A}(\mathcal{C}) \vert =\vert \mathcal{A}(\mathcal{C})^{*} \vert =\vert \mathcal{C}^{\perp}\vert=\frac{p^{12mp^s}}{\vert \mathcal{C} \vert}=\frac{p^{12mp^s}}{p^{3m(2p^s-t)}}=p^{3m(2p^s+t)}.
			\end{equation*}
			Thus, $\langle \phi_6(x), u(x^3-\alpha_0)^{2p^s-\ell}  \rangle=\mathcal{A}(\mathcal{C})$. 
			Let $z(x)=\sum\limits_{\kappa}^{}(z_{0\kappa}x^2+z_{1\kappa}x+z_{2\kappa})(x^3-\alpha_0)^{\kappa}$ where $z_{0\kappa},  z_{1\kappa}, z_{2\kappa} \in \mathbb{F}_{p^m}$ and $z_{00}x^2+z_{10}x+z_{20} \neq 0$. Here $\kappa \leq\Im-t-1=p^s-t-1$.
			\begin{align*}
				\phi_6(x)&=(x^3-\alpha_0)^{p^s}-2\alpha_2 u-uz(x)\\
				&=(x^3-\alpha_0)^{p^s}-2\alpha_2 u-u\sum\limits_{\kappa=0}^{p^s-t-1}(z_{0\kappa}x^2+z_{1\kappa}x+z_{2\kappa})(x^3-\alpha_0)^{\kappa}.
			\end{align*}
			Then $ \mathcal{A}(\mathcal{C})^{*}=\langle \phi^*_6(x) ,  (x^3-\alpha_0^{-1})^{2p^s-\ell} \rangle$,  where $\phi^*_6(x)=(-\alpha_0)^{p^s}(x^3-\alpha_0^{-1})^{p^s}-2\alpha_2 u x^{3p^s}- u\sum \limits_{\kappa=0}^{p^s-t-1}(z_{2\kappa}x^2+z_{1\kappa}x+z_{0\kappa})(-\alpha_0)^{\kappa}(x^3-\alpha_0^{-1})^{\kappa} x^{3p^s-2\kappa-1} $.
		\end{enumerate}
	\end{proof}
	\begin{theorem}
		If $\mathcal{C}=\langle (x^3-\alpha_0)^\ell +u(x^3-\alpha_0)^t z(x),  u (x^3-\alpha_0)^{\mu} \rangle $,  where $0\leq  \ell \leq 2p^s-1$,  $0\leq  t < \ell  $ and $z(x)$ is 0 or $z(x)$ is an invertible, then  the dual code $\mathcal{C}^\perp $ associate to the ideal $\mathcal{A}(\mathcal{C})^*$ is determined as follows:
		\begin{enumerate}
			\item If $z(x)$ is 0 and $1\leq \ell\leq p^s $,  then $ \mathcal{A}(\mathcal{C})^{*} = \langle \psi_1^*(x), u(x^3-\alpha_0^{-1})^{2p^s-\ell }  \rangle$,  where $ \psi_1^*(x)= (-\alpha_0)^{2p^s-\mu}(x^3-\alpha_0^{-1})^{2p^s-\mu}-2\alpha_2 u (-\alpha_0)^{p^s-\mu}(x^3-\alpha_0^{-1})^{p^s-\mu}x^{3p^s} $.
			\item If $z(x)$ is an invertible and $\ell \neq p^s+t$ then $ \mathcal{A}(\mathcal{C})^{*} = \langle \psi_2^*(x), u(x^3-\alpha_0^{-1})^{2p^s-\ell }  \rangle$,  where $ \psi_2^*(x)= (-\alpha_0)^{2p^s-\mu}(x^3-\alpha_0^{-1})^{2p^s-\mu}-2\alpha_2 u (-\alpha_0)^{p^s-\mu}(x^3-\alpha_0^{-1})^{p^s-\mu}x^{3p^s}-u(-\alpha_0)^{2p^s+t-\ell-\mu}(x^3-\alpha_0^{-1})^{2p^s+t-\ell-\mu}\\
		\sum\limits_{\kappa=0}^{\mu-t-1}(z_{2\kappa}x^2+z_{1\kappa}x+z_{0\kappa})(-\alpha_0)^{\kappa}(x^3-\alpha_0^{-1})^{\kappa} x^{3\ell-3t-3\kappa-2} $.
			\item If $z(x)$ is an invertible and $\ell=p^s+t$ then $ \mathcal{A}(\mathcal{C})^{*} = \langle \psi_4^*(x), u(x^3-\alpha_0^{-1})^{2p^s-\ell }  \rangle$,  where $ \psi_4^*(x)= (-\alpha_0)^{2p^s-\mu}(x^3-\alpha_0^{-1})^{2p^s-\mu}-2\alpha_2 u (-\alpha_0)^{p^s-\mu}(x^3-\alpha_0^{-1})^{p^s-\mu}x^{3p^s}-u(-\alpha_0)^{2p^s+t-\ell-\mu}(x^3-\alpha_0^{-1})^{2p^s+t-\ell-\mu} \\ \sum\limits_{\kappa=0}^{\mu-t-1}(z_{2\kappa}x^2+z_{1\kappa}x+z_{0\kappa})(-\alpha_0)^{\kappa}(x^3-\alpha_0^{-1})^{\kappa} x^{3\ell-3t-3\kappa-2} $.

		\end{enumerate}  
	\end{theorem}
	\begin{proof}
		\begin{enumerate}
			\item If $z(x)$ is 0 and $1\leq \ell\leq p^s $,  then $\mathcal{C}=\langle (x^3-\alpha_0)^\ell, u (x^3-\alpha_0)^{\mu}\rangle $. %If we choose $\rho=2p^s-\mu $,  $\xi=p^s-\mu$ and $r(x)=-2\alpha_2$. 
            By Lemma \ref{4.4}, $\psi_1(x)=(x^3-\alpha_0)^{2p^s-\mu}-2\alpha_2 u(x^3-\alpha_0)^{p^s-\mu} \in \mathcal{A}(\mathcal{C})$. Therefore, $\langle \psi_1(x),  u(x^3-\alpha_0)^{2p^s-\ell} \rangle \subseteq \mathcal{A}(\mathcal{C})$. Furthermore, 
			\begin{equation*}
				p^{3m(\ell+\mu)}=\vert \langle \psi_1(x), u(x^3-\alpha_0)^{2p^s-\ell}  \rangle \vert\leq \vert \mathcal{A}(\mathcal{C}) \vert =\vert \mathcal{A}(\mathcal{C})^{*} \vert =\vert \mathcal{C}^{\perp}\vert=\frac{p^{12mp^s}}{\vert \mathcal{C} \vert}=\frac{p^{12mp^s}}{p^{3m(4p^s-\ell-\mu)}}=p^{3m(\ell+\mu)}.
			\end{equation*}
			Thus, $\langle \psi_1(x), u(x^3-\alpha_0)^{2p^s-\ell}  \rangle=\mathcal{A}(\mathcal{C})$ and  $ \mathcal{A}(\mathcal{C})^{*} = \langle \psi_1^*(x), u(x^3-\alpha_0^{-1})^{2p^s-\ell }  \rangle$,  where $ \psi_1^*(x)= (-\alpha_0)^{2p^s-\mu}(x^3-\alpha_0^{-1})^{2p^s-\mu}-2\alpha_2 u (-\alpha_0)^{p^s-\mu}(x^3-\alpha_0^{-1})^{p^s-\mu}x^{3p^s} $.
			\item If $z(x)$ is an invertible and $\ell \neq p^s+t$, then $\mathcal{C}=\langle (x^3-\alpha_0)^\ell +u(x^3-\alpha_0)^t z(x),  u (x^3-\alpha_0)^{\mu} \rangle $.
			\begin{itemize}
				\item Let $1\leq \ell< p^s+t$. 
				%If we choose $\rho=2p^s-\mu $,  $\xi=p^s-\mu$ and $r(x)=-2\alpha_2-(x^3-\alpha_0)^{p^s-\ell+t}z(x)$ then 
            By Lemma \ref{4.4},  $\psi_2(x)=(x^3-\alpha_0)^{2p^s-\mu}+u(x^3-\alpha_0)^{p^s-\mu}[-2\alpha_2-(x^3-\alpha_0)^{p^s-\ell+t}z(x) ]\in \mathcal{A}(\mathcal{C})$. Therefore, $\langle \psi_2(x),  u(x^3-\alpha_0)^{2p^s-\ell} \rangle \subseteq \mathcal{A}(\mathcal{C})$. Furthermore, 
				\begin{equation*}
					p^{3m(\ell+\mu)}=\vert \langle \psi_2(x), u(x^3-\alpha_0)^{2p^s-\ell}  \rangle \vert\leq \vert \mathcal{A}(\mathcal{C}) \vert =\vert \mathcal{A}(\mathcal{C})^{*} \vert =\vert \mathcal{C}^{\perp}\vert=\frac{p^{12mp^s}}{\vert \mathcal{C} \vert}=\frac{p^{12mp^s}}{p^{3m(4p^s-\ell-\mu)}}=p^{3m(\ell+\mu)}.
				\end{equation*}
				Thus, $\langle \psi_2(x), u(x^3-\alpha_0)^{2p^s-\ell}  \rangle=\mathcal{A}(\mathcal{C})$.
				Let $z(x)=\sum\limits_{\kappa}^{}(z_{0\kappa}x^2+z_{1\kappa}x+z_{2\kappa})(x^3-\alpha_0)^{\kappa}$, where $z_{0\kappa},  z_{1\kappa}, z_{2\kappa} \in \mathbb{F}_{p^m}$ and $z_{00}x^2+z_{10}x+z_{20} \neq 0$. Here $\kappa \leq\mu-t-1$. Then 
				\begin{align*}
					\psi_2(x)&=(x^3-\alpha_0)^{2p^s-\mu}+u(x^3-\alpha_0)^{p^s-\mu}[-2\alpha_2-(x^3-\alpha_0)^{p^s-\ell+t}z(x) ]\\
					&=(x^3-\alpha_0)^{2p^s-\mu}-2\alpha_2 u(x^3-\alpha_0)^{p^s-\mu}-u(x^3-\alpha_0)^{2p^s-\ell-\mu+t}\sum\limits_{\kappa=0}^{\mu-t-1}(z_{0\kappa}x^2+z_{1\kappa}x+z_{2\kappa})(x^3-\alpha_0)^{\kappa}.
				\end{align*}
				Thus,  $ \mathcal{A}(\mathcal{C})^{*} = \langle \psi_2^*(x), u(x^3-\alpha_0^{-1})^{2p^s-\ell }  \rangle$,  where $ \psi_2^*(x)= (-\alpha_0)^{2p^s-\mu}(x^3-\alpha_0^{-1})^{2p^s-\mu}-2\alpha_2 u (-\alpha_0)^{p^s-\mu}(x^3-\alpha_0^{-1})^{p^s-\mu}x^{3p^s}-u(-\alpha_0)^{2p^s+t-\ell-\mu}(x^3-\alpha_0^{-1})^{2p^s+t-\ell-\mu}
				\sum\limits_{\kappa=0}^{\mu-t-1}(z_{2\kappa}x^2+z_{1\kappa}x+z_{0\kappa})(-\alpha_0)^{\kappa}(x^3-\alpha_0^{-1})^{\kappa} x^{3\ell-3t-3\kappa-2} $.
				\item Let $p^s+t< \ell \leq2p^s-1 $. 
    %If we choose $\rho=2p^s-\mu $,  $\xi=2p^s+t-\ell-\mu$ and $r(x)=-2\alpha_2(x^3-\alpha_0)^{\ell-p^s-t}z(x)$ then 
            By Lemma \ref{4.4},  $\psi_3(x)=(x^3-\alpha_0)^{2p^s-\mu}+u(x^3-\alpha_0)^{2p^s+t-\ell-\mu}[-2\alpha_2(x^3-\alpha_0)^{\ell-p^s-t}-z(x) ]\in \mathcal{A}(\mathcal{C})$. Therefore, $\langle \psi_3(x),  u(x^3-\alpha_0)^{2p^s-\ell} \rangle \subseteq \mathcal{A}(\mathcal{C})$. Furthermore, 
				\begin{equation*}
					p^{3m(\ell+\mu)}=\vert \langle \psi_3(x), u(x^3-\alpha_0)^{2p^s-\ell}  \rangle \vert\leq \vert \mathcal{A}(\mathcal{C}) \vert =\vert \mathcal{A}(\mathcal{C})^{*} \vert =\vert \mathcal{C}^{\perp}\vert=\frac{p^{12mp^s}}{\vert \mathcal{C} \vert}=\frac{p^{12mp^s}}{p^{3m(4p^s-\ell-\mu)}}=p^{3m(\ell+\mu)}.
				\end{equation*}
				Thus, $\langle \psi_3(x), u(x^3-\alpha_0)^{2p^s-\ell}  \rangle=\mathcal{A}(\mathcal{C})$. Let $z(x)=\sum\limits_{\kappa}^{}(z_{0\kappa}x^2+z_{1\kappa}x+z_{2\kappa})(x^3-\alpha_0)^{\kappa}$, where $z_{0\kappa},  z_{1\kappa}, z_{2\kappa} \in \mathbb{F}_{p^m}$ and $z_{00}x^2+z_{10}x+z_{20} \neq 0$. Here $\kappa \leq\mu-t-1$. Then 
				\begin{align*}
					\psi_3(x)&=(x^3-\alpha_0)^{2p^s-\mu}+u(x^3-\alpha_0)^{2p^s+t-\ell-\mu}-[2\alpha_2(x^3-\alpha_0)^{\ell-p^s-t}-z(x)]\\
					&=(x^3-\alpha_0)^{2p^s-\mu}-2\alpha_2 u(x^3-\alpha_0)^{p^s-\mu}-u(x^3-\alpha_0)^{2p^s+t-\ell-\mu}\sum\limits_{\kappa=0}^{\mu-t-1}(z_{0\kappa}x^2+z_{1\kappa}x+z_{2\kappa})(x^3-\alpha_0)^{\kappa}.
				\end{align*}
				Thus,  $ \mathcal{A}(\mathcal{C})^{*} = \langle \psi_3^*(x), u(x^3-\alpha_0^{-1})^{2p^s-\ell }  \rangle$,  where $ \psi_3^*(x)= (-\alpha_0)^{2p^s-\mu}(x^3-\alpha_0^{-1})^{2p^s-\mu}-2\alpha_2 u(-\alpha_0)^{p^s-\mu}(x^3-\alpha_0^{-1})^{p^s-\mu}x^{3p^s}-u(-\alpha_0)^{2p^s+t-\ell-\mu}(x^3-\alpha_0^{-1})^{2p^s+t-\ell-\mu}
				\sum\limits_{\kappa=0}^{\mu-t-1}(z_{2\kappa}x^2+z_{1\kappa}x+z_{0\kappa})(-\alpha_0)^{\kappa}(x^3-\alpha_0^{-1})^{\kappa} x^{3\ell-3t-3\kappa-2} $.
			\end{itemize}
			\item Let $z(x)$ be an invertible and $\ell=p^s+t$. 
   %If we choose $\rho=2p^s-\mu $,  $\xi=p^s-\mu$ and $r(x)=-2\alpha_2-(x^3-\alpha_0)^{p^s-\ell+t}z(x)$ then 
   By Lemma \ref{4.4}, $\psi_4(x)=(x^3-\alpha_0)^{2p^s-\mu}+u(x^3-\alpha_0)^{p^s-\mu}[-2\alpha_2-(x^3-\alpha_0)^{p^s-\ell+t}z(x) ]\in \mathcal{A}(\mathcal{C})$. Therefore, $\langle \psi_3(x),  u(x^3-\alpha_0)^{2p^s-\ell} \rangle \subseteq \mathcal{A}(\mathcal{C})$. Furthermore, 
			\begin{equation*}
				p^{3m(\ell+\mu)}=\vert \langle \psi_4(x), u(x^3-\alpha_0)^{2p^s-\ell}  \rangle \vert\leq \vert \mathcal{A}(\mathcal{C}) \vert =\vert \mathcal{A}(\mathcal{C})^{*} \vert =\vert \mathcal{C}^{\perp}\vert=\frac{p^{12mp^s}}{\vert \mathcal{C} \vert}=\frac{p^{12mp^s}}{p^{3m(4p^s-\ell-\mu)}}=p^{3m(\ell+\mu)}.
			\end{equation*}
			Thus, $\langle \psi_4(x), u(x^3-\alpha_0)^{2p^s-\ell} \rangle=\mathcal{A}(\mathcal{C})$. 
			Let $z(x)=\sum\limits_{\kappa}^{}(z_{0\kappa}x^2+z_{1\kappa}x+z_{2\kappa})(x^3-\alpha_0)^{\kappa}$, where $z_{0\kappa},  z_{1\kappa}, z_{2\kappa} \in \mathbb{F}_{p^m}$ and $z_{00}x^2+z_{10}x+z_{20} \neq 0$. Here $\kappa \leq\mu-t-1$. Then 
			\begin{align*}
				\psi_4(x)&=(x^3-\alpha_0)^{2p^s-\mu}+u(x^3-\alpha_0)^{p^s-\mu}[-2\alpha_2-(x^3-\alpha_0)^{p^s-\ell+t}z(x) ]\\
				&=(x^3-\alpha_0)^{2p^s-\mu}-2\alpha_2 u(x^3-\alpha_0)^{p^s-\mu}-u(x^3-\alpha_0)^{2p^s-\ell-\mu+t}\sum\limits_{\kappa=0}^{\mu-t-1}(z_{0\kappa}x^2+z_{1\kappa}x+z_{2\kappa})(x^3-\alpha_0)^{\kappa}.
			\end{align*}
			Then $ \mathcal{A}(\mathcal{C})^{*} = \langle \psi_4^*(x), u(x^3-\alpha_0^{-1})^{2p^s-\ell }  \rangle$,  where $ \psi_4^*(x)= (-\alpha_0)^{2p^s-\mu}(x^3-\alpha_0^{-1})^{2p^s-\mu}-2\alpha_2 u (-\alpha_0)^{p^s-\mu}(x^3-\alpha_0^{-1})^{p^s-\mu}x^{3p^s}-u(-\alpha_0)^{2p^s+t-\ell-\mu}(x^3-\alpha_0^{-1})^{2p^s+t-\ell-\mu}
			\sum\limits_{\kappa=0}^{\mu-t-1}(z_{2\kappa}x^2+z_{1\kappa}x+z_{0\kappa})(-\alpha_0)^{\kappa}(x^3-\alpha_0^{-1})^{\kappa} x^{3\ell-3t-3\kappa-2} $.
		\end{enumerate}
	\end{proof}

\section{  $\alpha=\alpha_1 +\alpha_4 uv $ is non-cube in $\mathcal{R}$}
%where $\alpha_1 ,\alpha_4 \in  \mathbb{F}^*_{p^m}$ }

Let  $\alpha=\alpha_1 +\alpha_4 uv $ be a non-cube in $\mathcal{R}$ where $\alpha_1, \alpha_4 \in \mathbb{F}^*_{p^m}$. Since $\alpha_1 \neq 0$, $\alpha_1 + \alpha_4 uv $ is a unit in $\mathcal{R}$ and $(\alpha_1 + \alpha_4 uv)$-constacyclic codes of length $3p^s$ over $\mathcal{R}$ are ideals of the quotient ring $R_{\alpha_1, \alpha_4} =\frac{\mathcal{R}[x]}{\langle x^{3p^s}-(\alpha_1 + \alpha_4 uv) \rangle}$.

\begin{proposition}\label{9.1}
        $\alpha=\alpha_1 +\alpha_4 uv$ is non-cube in $\mathcal{R}$ if and only if  $\alpha_1$ is non-cube in $\mathbb{F}_{p^m}$.
    \end{proposition}
   % \hl{This proof only proves one direction. Need to prove the other direction as well.}
   % \colorbox{green}{Proof of another direction added. }
   
       \begin{proof}
        Suppose $\alpha=\alpha_1 + \alpha_4 uv$ is non-cube in $\mathcal{R}$ and assume $\alpha_1= \alpha_1^{\prime^{3}}\in \mathbb{F}_{p^m}$. Then $\alpha_1 + \alpha_2 u + \alpha_3 v + \alpha_4 uv=(\alpha_1^{\prime} +\alpha_2^{\prime}u+ \alpha_3^{\prime} v + \alpha_4^{\prime} uv)^3$, where $\alpha_2^{\prime}=0,  \alpha_3^{\prime}=0$ and $\alpha_4^{\prime}=3^{-1} \alpha_1^{\prime^{-2}}\alpha_4$,  a contradiction. 
        
        Conversely, 
        %\hlc[cyan]{we have to prove if  $\alpha_1$ is not a square in $\mathbb{F}_{p^m}$ then $\alpha$ is not a square in $\mathcal{R}$. Now consider its contrapositive. i.e., if $\alpha$ is a square in $\mathcal{R}$ then $\alpha_1$ is a square in $\mathbb{F}_{p^m}$}.
        assume that $\alpha$ is cube in $\mathcal{R}$. Then there exists $\alpha_1^{\prime} +\alpha_2^{\prime}u+ \alpha_3^{\prime} v + \alpha_4^{\prime} uv$, where $\alpha_1^{\prime}, \alpha_2^{\prime}, \alpha_3^{\prime}, \alpha_4^{\prime} \in \mathbb{F}_{p^m}$, such that 
        \begin{align*}
            \alpha=&(\alpha_1^{\prime} +\alpha_2^{\prime}u+ \alpha_3^{\prime} v + \alpha_4^{\prime} uv)^3\\
            =&\alpha_1^{\prime^{3}}+(3\alpha_1^{\prime^{2}}\alpha_2^{\prime})u+(3\alpha_1^{\prime^{2}}\alpha_3^{\prime})v+(3\alpha_1^{\prime^{2}}\alpha_4^{\prime}+6\alpha_1^{\prime}\alpha_2^{\prime}\alpha_3^{\prime})uv.
        \end{align*}
        Thus, $\alpha_1 + \alpha_4 uv=\alpha_1^{\prime^{3}}+(3\alpha_1^{\prime^{2}}\alpha_2^{\prime})u+(3\alpha_1^{\prime^{2}}\alpha_3^{\prime})v+(3\alpha_1^{\prime^{2}}\alpha_4^{\prime}+6\alpha_1^{\prime}\alpha_2^{\prime}\alpha_3^{\prime})uv$. By comparing coefficients, we have $\alpha_1= \alpha_1^{\prime^{3}} \in \mathbb{F}_{p^m}$.
    \end{proof}
\begin{remark}
   We have $\alpha_1={\alpha_0}^{p^{s}}$. By Lemma \ref{9.1}, $\alpha$ is non-cube in $\mathcal{R}$ if and only if $\alpha_0$ is non-cube in $\mathbb{F}_{p^m}$.
\end{remark}
\begin{proposition}\label{9.2}
    Any non zero linear polynomial $c_1x^2+c_2x+c_3 \in \mathbb{F}_{p^m}[x]$ is a unit in $R_{\alpha_1, \alpha_4}$.
\end{proposition}
\begin{proof}
        Consider a non zero polynomial  $g(x)=c_1x^2+c_2x+c_3$ in $\mathbb{F}_{p^m}[x]$ i.e., $c_1,c_2,c_3 \in \mathbb{F}_{p^m}$ are not all zeros. 
        \begin{itemize}
            \item If $deg(g(x))=0$. i.e., $c_1=c_2=0, c_3\neq 0$. Thus, $g(x)=c_3$ is a unit.
            \item If $deg(g(x))=1$. i.e., $c_1=0, c_2\neq 0$ and $g(x)=c_2x+c_3$. 
            In $R_{\alpha_1, \alpha_4}$ we have 
            \begin{align*}
                (x-c_3)^{p^{s}}(x^2+c_3x+c_3^2)^{p^{s}}&=(x^3-c_3^2)^{p^{s}}\\
                &=x^{3p^{s}}-c_3^{2p^{s}}\\
                 &=\alpha_1+ \alpha_3 v + \alpha_4 uv -c_3^{3p^{s}}\\
                 &=(\alpha_1-c_3^{3p^{s}})+\alpha_4 uv.
            \end{align*} 
        Not that $\alpha_1-c_2^{3p^{s}}$ is unit in $\mathbb{F}_{p^m}$ as $\alpha_1$ is non-cube in $\mathbb{F}_{p^m}$. Hence, $(\alpha_1-c_2^{3p^{s}})+\alpha_4 uv$ is a unit in $\mathcal{R}$. Therefore, 
        \begin{center}
            $(x-c_3)^{-1}=(x-c_3)^{p^{s-1}}(x^2+c_3x+c_3^2)^{p^{s}}(\alpha_1-c_3^{3p^{s}}+\alpha_4 uv)^{-1}$.
        \end{center}
        For any $c_2\neq 0$ in $\mathbb{F}_{p^m}$,  we have 
        \begin{align*}
            (c_2x+c_3)^{-1}&=c_2^{-1}(x-c_2^{-1}(-c_3))^{-1}\\
            &=c_2^{-1}(x-c_2^{-1}(-c_3))^{p^{s-1}}(x^2+c_2^{-1}(-c_3)x+(c_2^{-1}(-c_3))^2)^{p^{s}}(\alpha_1-c_1^{-3p^{s}}c_2^{3p^{s}}+\alpha_4 uv)^{-1}.
        \end{align*}

        \item If $deg(g(x))=2$. i.e., $c_1\neq 0$ and $g(x)=c_1x^2+c_2x+c_3$. In $R_{\alpha_1, \alpha_4}$ we have 
    \begin{align*}
        g^{-1}(x)=&(c_1x^2+c_2x+c_3)^{-1}\\
        =&c^{-1}_1(x^2+c^{-1}_1c_2x+c^{-1}_1c_3)^{-1}\\
        =&c^{-1}_1(x^2+c^{-1}_1c_2x+c^{-1}_1c_3)^{p^s-1}(x^2+c^{-1}_1c_2x+c^{-1}_1c_3)^{-p^s}(x-c^{-1}_1c_2)^{-p^s}(x-c^{-1}_1c_2)^{p^s}\\
        =&c^{-1}_1(x^2+c^{-1}_1c_2x+c^{-1}_1c_3)^{p^s-1}(x-c^{-1}_1c_2)^{p^s}\big[(x^2+c^{-1}_1c_2x+c^{-1}_1c_3)(x-c^{-1}_1c_2)\big]^{-p^s}\\
        =&c^{-1}_1(x^2+c^{-1}_1c_2x+c^{-1}_1c_3)^{p^s-1}(x-c^{-1}_1c_2)^{p^s}\big[x^3+(c^{-1}_1c_3-c^{-2}_1c^{2}_2)x-c^{-2}_1c_2c_3\big]^{-p^s}\\
        =&c^{-1}_1(x^2+c^{-1}_1c_2x+c^{-1}_1c_3)^{p^s-1}(x-c^{-1}_1c_2)^{p^s}\big[x^{3p^s}+(c^{-1}_1c_3-c^{-2}_1c^{2}_2)^{p^s}x^{p^s}-(c^{-2}_1c_2c_3)^{p^s}\big]^{-1}\\
        =&c^{-1}_1(x^2+c^{-1}_1c_2x+c^{-1}_1c_3)^{p^s-1}(x-c^{-1}_1c_2)^{p^s} \big[ \alpha_1 + \alpha_4 uv+(c^{-1}_1c_3-c^{-2}_1c^{2}_2)^{p^s}x^{p^s}-(c^{-2}_1c_2c_3)^{p^s}\big]^{-1}\\
        =&c^{-1}_1(x^2+c^{-1}_1c_2x+c^{-1}_1c_3)^{p^s-1}(x-c^{-1}_1c_2)^{p^s}\big[ (\alpha_0-c^{-2}_1c_2c_3 +(c^{-1}_1c_3-c^{-2}_1c^{2}_2)x)^{p^s} +\alpha_4 uv\big]^{-1}
    \end{align*}
So, $g(x)$ is invertible if and only if $\alpha_0-c^{-2}_1c_2c_3 +(c^{-1}_1c_3-c^{-2}_1c^{2}_2)x$ is invertible. Suppose $\alpha_0-c^{-2}_1c_2c_3 +(c^{-1}_1c_3-c^{-2}_1c^{2}_2)x=0$. Then $\alpha_0-c^{-2}_1c_2c_3=0$ and  $c^{-1}_1c_3-c^{-2}_1c^{2}_2=0.$ This implies that $\alpha_0=(c_1^{-1}c_2)^{3}$, a contradiction as $\alpha_0$ is non-cube. 
 \end{itemize}
 Therefore, any non zero polynomial $c_1x^2+c_2x+c_3 \in \mathbb{F}_{p^m}[x]$ is a unit in $R_{\alpha_1, \alpha_4}$.
       
    \end{proof}
\begin{lemma}\label{9.3}
    In $\mathcal{R}$,  $\langle uv \rangle= \langle (x^3-\alpha_0)^{p^{s}} \rangle$, and $x^3-\alpha_0$ is a nilpotent element with nilpotency index $2p^s$.
\end{lemma}
\begin{proof}
    Let $x^3-\alpha_0 \in \mathcal{R}$. As $p$ is the characteristic and $p\vert\binom{p^s}{i}$,  for every $i$, $1\leq i\leq p^s-1$,  we have 
    \begin{align*}
        (x^3-\alpha_0)^{p^{s}}&=x^{3p^{s}}-\alpha_0^{p^{s}}+\sum_{i=1}^{p^s-1}(-1)^{p^{s}-i}\binom{p^s}{i}(x^{3})^{i}(\alpha_0)^{p^{s}-i}\\
        &=x^{3p^{s}}-\alpha_0^{p^{s}}\\
        &=x^{3p^{s}}-\alpha_1\\
        &=\alpha_4 uv.
    \end{align*}
    Since $\alpha_4 \neq 0$,  $\alpha_4$ is a unit in $R_{\alpha_1, \alpha_4}$.
    Thus, $\langle uv \rangle =\langle (x^3-\alpha_0)^{p^{s}} \rangle$. Hence,  the nilpotency index of $x^3-\alpha_0$ is $2p^s$.
\end{proof}

\begin{lemma}\label{9.4}
		The ring $R_{\alpha_1, \alpha_4}$ is a local ring with the unique maximal ideal  $\langle (x^3-\alpha_0), u, v \rangle $, but $R_{\alpha_1, \alpha_4}$ is not a chain ring.
	\end{lemma}
	\begin{proof}
		Let $h(x)$ be a polynomial in $R_{\alpha_1, \alpha_4}$. It can be expressed as  $h(x)=h_1(x)+uh_2(x)+vh_3(x)+uvh_4(x)$,  where $h_i(x)$, $1\leq i\leq 4$,
		are polynomials of degree less than or equal to $3p^s-1$ over $\mathbb{F}_{p^m}$. Thus, $h(x)$ can be uniquely expressed 
		\begin{align*}
			h(x)=&\sum_{\ell=0}^{p^s-1}(a_{0\ell}x^{2}+b_{0\ell}x+c_{0\ell})(x^3-\alpha_0)^\ell+u\sum_{\ell=0}^{p^s-1}(a_{1\ell}x^2+b_{1\ell}x+c_{1\ell})(x^3-\alpha_0)^\ell\\
			&+v\sum_{\ell=0}^{p^s-1}(a_{2\ell}x^{2}+b_{2\ell}x+c_{2\ell})(x^3-\alpha_0)^\ell+uv\sum_{\ell=0}^{p^s-1}(a_{3\ell}x^{2}+b_{3\ell}x+c_{3\ell})(x^3-\alpha_0)^\ell,
		\end{align*} 
		where $a_{j\ell},  b_{j\ell}, c_{j\ell} \in \mathbb{F}_{p^m}$,  $j=0, 1, 2, 3$. By Lemma \ref{9.3},  we have $uv=\alpha_4^{-1}(x^3-\alpha_0)^{p^{s}}$. Thus, $h(x)$ can be written as
		\begin{align}\label{eqn 6.1}
			\notag h(x)=&(a_{00}x^2+b_{00}x+c_{00})+\sum_{\ell=1}^{p^s-1}(a_{0\ell}x^{2}+b_{0\ell}x+c_{0\ell})(x^3-\alpha_0)^\ell+u\sum_{\ell=0}^{p^s-1}(a_{1\ell}x^2+b_{1\ell}x+c_{1\ell})(x^3-\alpha_0)^\ell\\
			\notag &+v\sum_{\ell=0}^{p^s-1}(a_{2\ell}x^{2}+b_{2\ell}x+c_{2\ell})(x^3-\alpha_0)^\ell+\alpha_4^{-1}(x^3-\alpha_0)^{p^{s}}\sum_{\ell=0}^{p^s-1}(a_{3\ell}x^{2}+b_{3\ell}x+c_{3\ell})(x^3-\alpha_0)^\ell\\
		\notag	=&(a_{00}x^2+b_{00}x+c_{00})+(x^3-\alpha_0)\sum\limits_{\ell=1}^{2p^s-1}(a^{\prime}_{0\ell}x^2+b^{\prime}_{0\ell}x+c^{\prime}_{0\ell})(x^3-\alpha_0)^{\ell-1}+u\sum\limits_{\ell=0}^{p^s-1}(a_{1\ell}x^2+b_{1\ell}x+c_{1\ell})(x^3-\alpha_0)^\ell\\
   & +v\sum_{\ell=0}^{p^s-1}(a_{2\ell}x^{2}+b_{2\ell}x+c_{2\ell})(x^3-\alpha_0)^\ell.
\end{align} 
    As $u$, $v$ and $(x^3-\alpha_0)$ are nilpotent  elements in $R_{\alpha_1, \alpha_4}$,  the polynomial $h(x)$ is non-unit if and only if $a_{00}=b_{00}=c_{00}=0$. Hence $\langle (x^3-\alpha_0),  u, v \rangle $ represents the set of all non-units of $R_{\alpha_1, \alpha_4}$,  which shows that $R_{\alpha_1, \alpha_4}$ is a local ring, and its unique maximal ideal is given by $\langle (x^3-\alpha_0), u, v\rangle $. 
    Now we shall show that $\langle (x^3-\alpha_0),  u, v\rangle $ is not a principle ideal of $R_{\alpha_1, \alpha_4}$. Suppose $u \in \langle (x^3-\alpha_0), v \rangle $. Then there exist $\hbar(x)=\sum\limits_{\ell=0}^{2p^s-1}(\hbar_{0\ell}x^2+\hbar^{\prime}_{0\ell}x+\hbar^{\prime \prime}_{0\ell})(x^3-\alpha_0)^{\ell}+u\sum\limits_{\ell=0}^{p^s-1}(\hbar_{1\ell}x^2+\hbar^{\prime}_{1\ell}x+\hbar^{\prime \prime}_{1\ell})(x^3-\alpha_0)^\ell +v\sum\limits_{\ell=0}^{p^s-1}(\hbar_{2\ell}x^2+\hbar^{\prime}_{2\ell}x+\hbar^{\prime \prime}_{2\ell})(x^3-\alpha_0)^\ell$ and $\wp(x)=\sum\limits_{\ell=0}^{2p^s-1}(\wp_{0\ell}x^2+\wp^{\prime}_{0\ell}x+\wp^{\prime \prime}_{0\ell})(x^3-\alpha_0)^{\ell}+u\sum\limits_{\ell=0}^{p^s-1}(\wp_{1\ell}x^2+\wp^{\prime}_{1\ell}x+\wp^{\prime \prime}_{1\ell})(x^3-\alpha_0)^\ell +v\sum\limits_{\ell=0}^{p^s-1}(\wp_{2\ell}x^2+\wp^{\prime}_{2\ell}x+\wp^{\prime \prime}_{2\ell})(x^3-\alpha_0)^\ell \in R_{\alpha_1, \alpha_4}$ such that $u=(x^3-\alpha_0)\hbar(x)+v\wp(x)$. Hence, 
    \begin{align}\label{4}
     \notag   u=&(x^3-\alpha_0)\Bigg[\sum\limits_{\ell=0}^{2p^s-1}(\hbar_{0\ell}x^2+\hbar^{\prime}_{0\ell}x+\hbar^{\prime \prime}_{0\ell})(x^3-\alpha_0)^{\ell}+u\sum\limits_{\ell=0}^{p^s-1}(\hbar_{1\ell}x^2+\hbar^{\prime}_{1\ell}x+\hbar^{\prime \prime}_{1\ell})(x^3-\alpha_0)^\ell \\
     \notag &+v\sum\limits_{\ell=0}^{p^s-1}(\hbar_{2\ell}x^2+\hbar^{\prime}_{2\ell}x+\hbar^{\prime \prime}_{2\ell})(x^3-\alpha_0)^\ell\Bigg]
    +v\Bigg[\sum\limits_{\ell=0}^{2p^s-1}(\wp_{0\ell}x^2+\wp^{\prime}_{0\ell}x+\wp^{\prime \prime}_{0\ell})(x^3-\alpha_0)^{\ell}\\
    \notag &+u\sum\limits_{\ell=0}^{p^s-1}(\wp_{1\ell}x^2+\wp^{\prime}_{1\ell}x+\wp^{\prime \prime}_{1\ell})(x^3-\alpha_0)^\ell +v\sum_{\ell=0}^{p^s-1}(\wp_{2\ell}x^2+\wp^{\prime}_{2\ell}x+\wp^{\prime \prime}_{2\ell})(x^3-\alpha_0)^\ell \Bigg]\\
      \notag  =&(x^3-\alpha_0)\sum\limits_{\ell=0}^{2p^s-1}(\hbar_{0\ell}x^2+\hbar^{\prime}_{0\ell}x+\hbar^{\prime \prime}_{0\ell})(x^3-\alpha_0)^{\ell}+u(x^3-\alpha_0)\sum\limits_{\ell=0}^{p^s-1}(\hbar_{1\ell}x^2+\hbar^{\prime}_{1\ell}x+\hbar^{\prime \prime}_{1\ell})(x^3-\alpha_0)^\ell \\
      \notag  &+v(x^3-\alpha_0) \sum\limits_{\ell=0}^{p^s-1}(\hbar_{2\ell}x^2+\hbar^{\prime}_{2\ell}x+\hbar^{\prime \prime}_{2\ell})(x^3-\alpha_0)^\ell+v \sum\limits_{\ell=0}^{2p^s-1}(\wp_{0\ell}x^2+\wp^{\prime}_{0\ell}x+\wp^{\prime \prime}_{0\ell})(x^3-\alpha_0)^{\ell}\\
    &+uv\sum\limits_{\ell=0}^{p^s-1}(\wp_{1\ell}x^2+\wp^{\prime}_{1\ell}x+\wp^{\prime \prime}_{1\ell})(x^3-\alpha_0)^\ell.
    \end{align}
  From Equation \ref{4},
  \begin{equation*}
      (x^3-\alpha_0)\sum\limits_{\ell=0}^{p^s-1}(\hbar_{1\ell}x^2+\hbar^{\prime}_{1\ell}x+\hbar^{\prime \prime}_{1\ell})(x^3-\alpha_0)^\ell=1,
  \end{equation*}
which implies that $x^3-\alpha_0$ is invertible. This is a contradiction. Hence, $u \notin \langle (x^3-\alpha_0), v \rangle $. Similarly, we can show that $v \in \langle (x^3-\alpha_0), u \rangle $. Next, we show that $(x^3-\alpha_0)\notin \langle u, v \rangle$. Suppose $(x^3-\alpha_0)\in \langle u, v \rangle$. Then there exist $f(x),g(x)\in R_{\alpha_1, \alpha_4}$ such that, $(x^3-\alpha_0)=uf(x)+vg(x)$. Here the nilpotency index of $uf(x)+vg(x)$ is at most 3, whereas the nilpotency index of $(x^3-\alpha_0)$ is at least 6, which gives a contradiction. Therefore,  the maximal ideal $\langle (x^3-\alpha_0), u, v \rangle $ is not principal. By Proposition \ref{prop2.1}, $R_{\alpha_1, \alpha_4}$ is not a chain ring.
\end{proof}

\section{Conclusion}
This article examined the algebraic structures of $\alpha$-constacyclic codes of length $3p^s$ over $\mathcal{R}=\mathbb{F}_{p^m} + u\mathbb{F}_{p^m} + v\mathbb{F}_{p^m} +uv\mathbb{F}_{p^m}$, where $\alpha$ is a unit of $\mathcal{R}$. In the case of $\alpha=\beta^3$, where $\beta \in \mathcal{R}$, any $\alpha$-constacyclic code of length $3p^s$ is determined under two conditions $p^m \equiv 1 (\mod{3})$ and $p^m \equiv 2 (\mod{3})$.  In the case of a non-cube $\alpha$, we determined the structures of the ideals of $\mathcal{R}_{\alpha}=\frac{\mathcal{R}[x]}{\langle x^{3p^s}-\alpha \rangle}$ in three of the eight subcases. In the subcases of $\alpha=\alpha_1 + \alpha_3 v + \alpha_4 uv$ and $\alpha=\alpha_1 + \alpha_2 u + \alpha_3 v + \alpha_4 uv$, it is shown that the rings $\frac{\mathcal{R}[x]}{\langle x^{3p^s}-\alpha \rangle}$  are local rings with maximal ideal $\langle (x^3-\alpha_0), u\rangle $, but they are  not chain rings. Moreover, $\alpha$-constacyclic codes are classified into 4 distinct types of ideals of the ring $\frac{\mathcal{R}[x]}{\langle x^{3p^s}-\alpha \rangle}$. Also, the number of codewords and dual codes of such constacyclic codes are investigated. In the third case of a non-cube $\alpha=\alpha_1 + \alpha_4 uv$, we discussed essential properties of the ring $\frac{\mathcal{R}[x]}{\langle x^{3p^s}-\alpha \rangle}$ to analyze the structure of $(\alpha_1 + \alpha_4 uv)$-constacyclic codes.

\section*{Authors Contributions} All authors contributed equally to this manuscript.
\section*{Conflict of Interest} The authors declare that they have no conflict of interest.
\section*{Data Availability} The authors confirm that the data supporting the findings of this study are available within the article.

\end{document}